\documentclass[11pt]{amsart}
\usepackage{latexsym}
\usepackage{amssymb,amsmath,amsthm,amscd, amssymb,graphicx, enumerate, verbatim, url}
\usepackage[all]{xy}

\numberwithin{equation}{section}
\newtheorem{cor}[equation]{Corollary}
\newtheorem{lem}[equation]{Lemma}
\newtheorem{prop}[equation]{Proposition}
\newtheorem{thm}[equation]{Theorem}
\newtheorem{prob}[equation]{Problem}
\newtheorem{quest}[equation]{Question}

\newtheorem{Example}[equation]{Example}
\newenvironment{ex}{\begin{Example}\rm}{\end{Example}}
\newtheorem{remark}[equation]{Remark}
\newenvironment{rmk}{\begin{remark}\rm}{\end{remark}}

\def\co{\colon\thinspace}

\newcommand{\Int}{\mbox{Int}}

\newcommand{\vol}{\mbox{Vol}}

\newcommand{\Ric}{\mbox{Ric}}
\newcommand{\Iso}{\mbox{Iso}}

\newcommand{\cd}{\mbox{cd}}
\newcommand{\ab}{\mathrm{ab}}
\newcommand{\e}{\varepsilon}

\def\a{\alpha}
\def\G{\Gamma}
\def\g{\gamma}

\def\d{\partial}

\def\S1{\bf S^1}

\begin{document}

\abovedisplayskip=6pt plus3pt minus3pt
\belowdisplayskip=6pt plus3pt minus3pt

\title[Topology of open nonpositively curved manifolds]
{\bf Topology of open nonpositively curved manifolds}
\rm

\author{Igor Belegradek}

\address{Igor Belegradek\\School of Mathematics\\ Georgia Institute of
Technology\\ Atlanta, GA 30332-0160}\email{ib@math.gatech.edu}


\date{}
\begin{abstract} 
This is a survey of topological properties of open, 
complete nonpositively curved manifolds which may have
infinite volume. Topics include topology of ends, 
restrictions on the fundamental group, 
as well as a review of known examples. 
\end{abstract}
\maketitle

Among many monographs and surveys on aspects of nonpositive curvature,
none deals with topology of open complete nonpositively curved manifolds, and 
this paper aims to fill the void. 
Most of the material discussed here is not widely-known. 
A number of questions is posed, ranging from naive to hopelessly difficult.
Proofs are supplied when there is no explicit reference, and
as always, the author is solely responsible for mistakes.

This survey has a narrow focus and does not discuss 
topological properties of 
\begin{itemize} 
\item
hyperbolic $3$-manifolds~\cite{Thu-3d-notes, Kap-book, CanMcC}, 
\item
open negatively pinched manifolds~\cite{Bel-invent, BK-acta},
\item
non-Riemannian nonpositively
curved manifolds~\cite{Dav-book}, 
\item
compact nonpositively curved 
manifolds~\cite{FJO-survey, Lue-surv-asph}, 
\item
higher rank locally symmetric spaces~\cite{Mar-book, Ebe-book, BesFei02, Gel-vol-rk}
and their compactifications~\cite{BorJi},
\end{itemize}
which are covered in the above references.
We choose to work in the Riemannian setting which leads to some simplifications,
even though many results hold in a far greater generality, 
and the references usually point to the strongest results available.
Special attention is given to rank one manifolds, such as manifolds of negative
curvature. Nonpositively curved manifolds are aspherical so 
we focus on groups of finite cohomological 
dimension (i.e. fundamental groups of aspherical manifolds),
or better yet, groups of type {\it F\,} (i.e. the fundamental groups 
of compact aspherical manifolds with boundary). 

{\bf Conventions:} unless stated otherwise, manifolds are smooth, 
metrics are Riemannian, and
sectional curvature is denoted by $K$. 

{\bf Acknowledgments:} The author is grateful for NSF support (DMS-1105045).
Thanks are due to Jim Davis, Denis Osin, Yunhui Wu, and the referee for correcting misstatements 
in the earlier version.

\tableofcontents

\section{Flavors of negative curvature}
\label{sec: flavors of neg curv}

A {\it Hadamard manifold\,} is a connected simply-connected complete
manifold of $K\le 0$. By the Cartan-Hadamard theorem,
any Hadamard manifold is diffeomorphic to a Euclidean space. 
Thus any complete manifold of $K\le 0$
is the quotient of a Hadamard manifold by a discrete torsion-free
isometry group (torsion-freeness can be seen geometrically: any 
finite isometry group of a Hadamard manifold fixes
the circumcenter of its orbit, or topologically: a nontrivial finite group 
has infinite cohomological dimension so it
cannot act freely on a contractible manifold).

A Hadamard manifold is a {\it visibility} manifold if any two points at infinity
can be joined by a geodesic. A {\it flat\,} in $X$
is a convex subset isometric to a Euclidean space.
A {\it flat half plane} in $X$ is a convex subset isometric to a Euclidean half plane.
An {\it infinitesimal flat\,} of a geodesic is the space of parallel Jacobi fields along
a geodesic. A geodesic is called {\it rank one\,} if its infinitesimal flat is one-dimensional. 
A complete manifold of $K\le 0$ {\it has rank one\,} if it contain a complete  (i.e. defined for all times)
geodesic of rank one. 

The following conditions on a Hadamard manifold $X$ represent
various manifestations of negative curvature:
\begin{itemize}
\label{form: neg curv conditions}
\item[(1)]  $K$ is bounded above by a negative constant,
\item[(2)]  $X$ is Gromov hyperbolic,
\item[(3)]  $X$ is a visibility manifold,
\item[(4)]  $K<0$,
\item[(5)]  $X$ contains no flat half plane,
\item[(6)] $X$ contains a complete geodesic 
that does not bound a flat half plane.
\end{itemize}

The implications
$(1)\Rightarrow (4)\Rightarrow (5)\Rightarrow (6)$ and
$(1)\Rightarrow (2)\Rightarrow (3)\Rightarrow (5)\Rightarrow (6)$ are 
immediate from definitions except for 
$(2)\Rightarrow (3)$ which 
can be found in~\cite[Lemma VIII.3.2]{BriHae}.

\begin{prop}
\label{prop: curv flavors}
Every other implication fails in dimension two.
\end{prop}

\begin{rmk}
There are of course analogs of (5) or (6) such as
``no flat strip'' or ``no flat plane'', see also
various axioms on~\cite{EbeOni},
but the above list is what comes up most often. 
\end{rmk}

\begin{rmk}
Kleiner~\cite{Kle-inf-flat} 
proved that a Hadamard manifold has rank one if and only if it contains a complete geodesic
that does not lie in a two-dimensional flat 
(see~\cite[Proposition IV.4.4]{Bal-book} for the case when 
$\Iso(X)$ satisfies the duality condition,
e.g. contains a lattice). 
Since (6) is intermediate between the two conditions, it
is equivalent to them.
\end{rmk}

\begin{proof}[Proof of Proposition~\ref{prop: curv flavors}]
That $(2)\nRightarrow (4)$ follows by doubling along the boundary 
any nonpositively curved compact surface of negative Euler characteristic 
whose metric is cylindrical near the boundary; the universal cover is hyperbolic
because hyperbolicity is a quasiisometry invariant, but it contains
a flat strip. 

To show that $(3)\nRightarrow (2)$ recall a consequence of 
the Gauss-Bonnet theorem, mentioned on~\cite[page 57]{BGS}, that
a $2$-dimensional Hadamard manifold $X$ is visibility if and only if 
for some $p\in X$ 
the total curvature of every sector bounded by two rays that start at $p$
is infinite. 
Any smooth nonpositive function $K$ 
on $[0,\infty)$ can be realized as the curvature of a rotationally
symmetric metric $dr^2+f(r)^2d\theta$ on $\mathbb R^2$, namely,  
$f$ is a unique solution of $f^{\prime\prime}+Kf=0$, 
$f(0)=1$, $f^\prime(0)=1$. Note that $f(r)\ge r$ by Sturm comparison
for ordinary differential equations, and 
the total curvature equals $2\pi\int_0^\infty Kfdr$.
Let ${I_i}$ be a sequence of disjoint compact subintervals 
of $(0,\infty)$ such that $I_i$ has length $i$, and let $K$
be a smooth negative function on $[0,\infty)$  that
equals $-\frac{1}{i}$ on $I_i$. The associated
rotationally symmetric metric has infinite total curvature on every
sector from the origin, 
so it is visibility, and it contains arbitrary large regions of
curvature $-\frac{1}{i}$, which contain triangles violating 
$\delta$-hyperbolicity for any $\delta$. 

That $(4)\nRightarrow (3)$ is also shown  
in~\cite[Example 5.10]{EbeOni}: modifying the
argument of the previous paragraph
yields a rotationally symmetric metric with $K<0$ and finite total curvature.

To see that $(6)\nRightarrow (5)$ start from a finite volume
complete hyperbolic metric on a punctured torus, and modify the metric
near the end to a complete metric of $K\le 0$ that is a cylinder outside
a compact set. The cylinder lifts to a flat half plane so (5) fails,
but any geodesic through a point of $K<0$ satisfies (6).

The other non-implications are formal consequence of the
above implications and non-implications. 
\end{proof}

\begin{rmk}
The non-implications in the above proof 
are justified by two dimensional
examples, and it seems that similar examples in higher dimensions
can be produced via iterated warping with $\cosh r$, i.e. replacing 
$X$ with $\mathbb R\times_{\cosh r} X$, which contains 
$X$ as a totally geodesic submanifold. For some non-implications we need
to insert the warped product as a convex subset in 
a closed manifold of $K\le 0$ following Ontaneda, see~\cite{Pha-hypersurf}. 
The warping clearly preserves conditions (1), (4), (5), (6).
\end{rmk}

\begin{quest}
Does the warping with $\cosh r$ preserve \textup{(2)} and \textup{(3)}?
\end{quest}

One wants to understand the relations among 
(1)--(6) when $\Iso(X)$ is large.

\begin{ex}
If $\Iso(X)$ is cocompact, then clearly $(1)\Leftrightarrow (4)$,
and also $(2)\Leftrightarrow (3)\Leftrightarrow (5)$; in fact, 
Gromov hyperbolicity 
is equivalent to uniform visibility for proper 
CAT($0$) spaces~\cite[Proposition III.3.1]{BriHae}, 
and visibility is equivalent to the non-existence of a 
$2$-flat for proper CAT($0$) spaces with cocompact 
isometry groups~\cite[Theorem II.9.33]{BriHae}.
For $2$-dimensional Hadamard manifolds 
with cocompact isometry group $(5)\Leftrightarrow (6)$
because if $X$ contains a flat half plane and has cocompact isometry group,
then $X$ is isometric to $\mathbb R^2$.
On the other hand, $(6)\nRightarrow (5)$ in dimensions $>2$
with examples given by the {\it Heintze manifolds\,}, obtained by
chopping of cusps of a finite volume
complete hyperbolic manifold, changing the metric near the cusp 
a metric of $K\le 0$ with totally geodesic flat boundary, and doubling
along the boundary; the boundary lifts to a flat of dimension $>1$,
while any geodesic through a point of $K<0$ has rank one.
\end{ex}

\begin{ex}
If $\dim(X)=2$ and $\Iso(X)$ contains a lattice,
then either $X$ is isometric to $\mathbb R^2$ or $X$ is 
visibility~\cite[Proposition 2.5]{Ebe-MemAMS-1979}
so that $(6)\Leftrightarrow (3)$.
\end{ex}

\begin{ex} If $\dim(X)>2$, then $(6)\nRightarrow (3)$; 
indeed Nguyen Phan~\cite{Pha-neg} and Wu~\cite{Wu-transl-length} show
that the universal covers of the finite volume manifolds of $K<0$
constructed in~\cite{Fuj-warp} are not visibility, 
so $(4)\nRightarrow (3)$.
\end{ex}

\begin{ex} Eberlein showed~\cite{Ebe-lattices-annals} 
that the ends of a finite volume complete manifold with bounded nonpositive curvature
and visibility universal cover are $\pi_1$-injectively embedded. This
is not the case for Buyalo's example~\cite{Buy} of a finite volume
complete $4$-manifold of $-1<K<0$; thus its universal cover is not visibility. 
\end{ex}

\begin{rmk} Let us compare $(3)$ and $(4)$ in higher dimensions. 
Non-visibility of $X$ can be checked by 
studying its two-dimensional totally geodesic submanifolds, 
cf.~\cite{Pha-neg, Wu-transl-length},
while proving visibility of $X$ gets much harder,
and indeed, it is a strong restriction on $X$ with
a variety of consequences for the geometry of horoballs, 
Tits boundary, and isometry groups. By contrast,
the condition ``$K<0$'' is easy to verify 
but has few implications.
Either condition implies that every complete geodesic in $X$
has rank one.  
\end{rmk}

\section{Manifolds of dimensions two and three}
 
A classification of open connected $2$-manifolds goes back to
Ker{\'e}kj{\'a}rt{\' o}~\cite{Ker-2-mfld}; 
see~\cite{Ric-2-mfld, Gol-2-mfld, Sie-2-mfld} for more recent accounts.

Richards~\cite[Theorem 3]{Ric-2-mfld} proved that every 
open surface is obtained from $S^2$ by removing a closed
totally disconnected subset, and then
removing a finite or countable family of disjoint closed 
disks and attaching handles or M\"obius bands along boundaries of the disks. 

\begin{rmk}
Any closed totally disconnected subset $T\subset S^2$ can 
be moved by an ambient homeomorphism to a subset
of the standard Cantor set. (Like any compact totally
disconnected metric space, $T$ is homeomorphic to a subset $Q$ of the Cantor set
in $S^2$. The homeomorphism type of a
planar surface is determined by the 
homeomorphism type of its
the space of ends~\cite[Theorem 1]{Ric-2-mfld}, 
which for $S^2\setminus T$, $S^2\setminus Q$ are identified with $T$, $Q$, respectively. 
Hence there is a homeomorphism $S^2\setminus T\to S^2\setminus Q$, which
extends to a homeomorphism of the end compactifications 
$S^2\to S^2$ mapping $T$ to $Q$).
\end{rmk}

The uniformization theorem equips any connected surface with a constant
curvature metric, which can be chosen hyperbolic for open surfaces:

\begin{thm}
Any open connected $2$-manifold admits a complete metric of constant negative curvature.
\end{thm}
\begin{proof}
Given an open manifold $M$, we equip it with a complete Riemannian metric, 
and pull the metric 
back to the universal cover $\widetilde M$, where 
$\pi_1(M)$ acts isometrically, and hence by
preserving the conformal class of the metric.
By the uniformization theorem
$\widetilde M$ is conformal to the hyperbolic plane
$\mathbf H^2$ or to $\mathbb C$. A self-diffeomorphism of $\mathbf H^2$ 
that preserve the conformal class of the hyperbolic metric
is an isometry of $\mathbf H^2$. 
A self-diffeomorphism of $\mathbb C$ that
preserves the conformal class of the Euclidean metric is
of the form $z\to az+b$ or $z\to a\bar z+b$, and hence it 
either is an isometry (i.e. $a\bar a=1$),
or its square has a fixed point 
(in fact, $z\to az+b$ fixes $\frac{b}{1-a}$
and similarly for the square of $z\to a\bar z+b$). The deck-transformation
$\pi_1(M)$-action is free, so it preserves the hyperbolic or the Euclidean 
metric. Thus $M$ admits a complete metric of constant curvature $-1$ or $0$. 
Finally, a flat open $2$-manifold is $\mathbb R^2$, an annulus, or a M\"obius band,
so it also admits a complete hyperbolic metric.
\end{proof}

Most of what is known on the geometrization of an
open $3$-manifold requires it to have finitely generated
fundamental group, or better yet to be the interior of a compact manifold. 

\begin{thm}
\label{thm: 3-mflds npc}
The interior of any compact aspherical $3$-manifold with nonempty boundary
admits a complete metric 
of $-1\le K\le 0$.
\end{thm}
\begin{proof}
We refer to~\cite{Kap-book} for the terminology used in this proof.
Let $N$ be a compact aspherical $3$-manifold. 
Asphericity implies that $N$ contains
no essential $2$-sphere or projective plane.
There is a decomposition of $N$ along incompressible $2$-sided
tori and Klein bottles into Seifert and atoroidal pieces,
see~\cite{BonSie1987}, which extends previous proofs by 
Johannson and Jaco-Shalen to the non-orientable case.
Each piece is $\pi_1$-injectively embedded and hence 
aspherical.

Atoroidal pieces contain no essential annuli or M\"obius
bands whose boundary circles lie on the components of
$\d N$ of zero Euler characteristic
(for the orientable case see e.g.~\cite[Lemma 1.16]{Hat-3mfld-notes}, and the
non-orientable case follows by passing to the orientation cover for
an essential annulus or M\"obius band stays essential in the orientation cover,
and a virtually Seifert piece is Seifert).

Define the pared manifold structure on every atoroidal piece N by letting
its parabolic locus $P$ be the union of boundary components of zero Euler
characteristic. Thurston's hyperbolization theorem~\cite[Theorem 1.43]{Kap-book} 
gives $(N,P)$ a geometrically finite hyperbolic structure whose parabolic subgroups are of
rank $2$ and bijectively corresponding to components of $P$. 
Every Seifert piece has a nonpositively curved metric
with totally geodesic (flat) boundary by Leeb~\cite{Lee-3mnf-npc}.
In the same paper Leeb gives a gluing procedure for identifying 
rank $2$ cusps and flat boundary components of Seifert pieces.
(Leeb assumes that hyperbolic pieces have boundary of zero Euler characteristic
but this does not matter since his gluing runs in the cusps, 
which in our case have rank $2$). The result is a complete nonpositively curved
metric on the interior of $N$.
\end{proof}

\begin{rmk}
A variation of the above proof identifies $N$ with a compact
locally convex $C^1$ manifold of $K\le 0$ with boundary, namely,
replace geometrically finite pieces with the $\e$-neighborhoods of their convex cores, 
where the cusps are chopped off, and the
metric at cusps are modified so that their boundaries are flat and
totally geodesic.
Thus the group $\pi_1(N)$ is CAT($0$), which was first noted 
in~\cite[Theorem 4.3]{Bri-3mfd-npc} with a slightly 
different proof.
\end{rmk}

\begin{rmk} 
More information on $3$-manifold groups can be found
in the survey~\cite{AFW-3mflds-grps}; in particular,
if $M$ is open aspherical $3$-manifold with finitely generated 
fundamental group, then the advances on the virtual Haken conjecture
imply that $\pi_1(M)$ virtually embeds into a right angled Artin 
group, which has a number of group-theoretic consequences, e.g. $\pi_1(M)$
is linear over $\mathbb Z$.
\end{rmk}

By Scott/Shalen compact core theorem 
any open aspherical $3$-manifold $M$ with finitely generated 
deformation retracts to a compact codimension zero submanifold
(to which Theorem~\ref{thm: 3-mflds npc} applies) but the topology
of $M$ is still a mystery.

A open $3$-manifold is called {\it tame\,} if it is homeomorphic
to the interior of a compact manifold. Marden's Tameness Conjecture
predicted tameness of every open complete $3$-manifold with $K\equiv -1$
and finitely generated fundamental group, and it was proved
by Agol~\cite{Ago-tame} and Gabai-Calegari~\cite{CalGab-tame}.
The following question is still open:

\begin{quest} Let $M$ be an open $3$-manifold  
with a complete metric of $K\le 0$ and
finitely generated fundamental group. Is $M$ tame? 
\end{quest}

The same question is open for manifolds of $K\le -1$ or 
$-1\le K\le 0$.
The answer is yes under any of the following assumptions: 
\begin{itemize}
\item 
$M$ admits a complete negatively pinched metric (due to 
Bowditch \cite{Bow-tame} who built on proofs 
of Marden's Tameness Conjecture by Agol, Calegari-Gabai, 
and Soma).
\item 
$M$ is a cover of the interior of a compact manifold
(as proved by Long-Reid in~\cite{Can-tamness-2008} that 
combines results of Simon with the proof of Tameness Conjecture;
here the assumption that $M$ has $K\le 0$ is not needed). 
\item $M$ 
has a complete metric of $-1\le K\le 0$ and
$\mathrm{Inj}\,\mathrm{Rad}\to 0$ (as proved by Schroeder~\cite[Appendix 2]{BGS}).
\end{itemize}
By contrast, there exits non-tame, open $3$-manifolds
with universal cover diffeomorphic to $\mathbb R^3$ and  
fundamental groups isomorphic to $\mathbb Z$~\cite{ScoTuc}
and $\mathbb Z\ast\mathbb Z$ \cite{FreGab}. 

\begin{ex}
The non-tame fake open solid torus in~\cite{ScoTuc} does not
admit a complete metric of $K\le 0$ because
$\mathbb R^2$-bundles over $S^1$ are the only open, complete 
$3$-manifolds of $K\le 0$ with infinite cyclic fundamental group.
(The isometric $\mathbb Z$-action in the 
universal cover must stabilize a geodesic or a horoball,
in the former case the nearest point projection to the geodesic
descends to an $\mathbb R^2$-bundle over $S^1$, while
in the latter case the quotient is the product 
of $\mathbb R$ and an open surface homotopy equivalent 
to a circle, which is an $\mathbb R$-bundle over $S^1$).
\end{ex}

\section{Towards a rough classification of discrete isometry groups}
\label{sec: rough}

In this section we sketch a higher-dimensional analog of 
the classification of open complete constant
curvature surfaces; details 
appear in Sections~\ref{sec: non-parabolic}--\ref{sec: anchored}.
We refer to~\cite{BGS, Ebe-book, BriHae} and Section~\ref{sec: flavors of neg curv} 
for background on Hadamard manifolds, adopt the following notations:
\begin{itemize}
\item
$X$ is a Hadamard manifold with ideal boundary $X(\infty)$, 
\vspace{1pt}
\item
$\Iso(X)$ is the isometry group of $X$,
\vspace{1pt}
\item
$\G$ is a subgroup of $\Iso(X)$,
\end{itemize}
and consider the following three classes of
complete manifolds of $K\le 0$ of the form $X/G$
where $\G\le\Iso(X)$ is discrete and torsion-free:
\begin{enumerate}
\item
$\G$ contains no parabolic elements.
\vspace{1pt}
\item
$\G$ contains a rank one element.
\vspace{1pt}
\item 
$\G$ fixes a point $\xi\in X(\infty)$, and the associated $\G$-action
on  the space $L_\xi$ of lines asymptotic to $\xi$ is
free and properly discontinuous.
\end{enumerate}
There are severe algebraic restrictions on $\G$ in the cases (1)-(2), and in the case
(3) the manifold $X/\G$ is diffeomorphic to the product of $\mathbb R$ and $L_\xi/\G$,
where $L_\xi$ is diffeomorphic to a Euclidean space. 

\begin{ex}
$\G$ satisfies (3) if it stabilizes a horoball. 
(Indeed, $L_\xi$ is equivariantly diffeomorphic to a horosphere, and
the discreteness of $\G$ implies that its action on 
$L_\xi$ is properly discontinuous).
\end{ex}

\begin{quest}
Suppose $X$ has rank one. Does every discrete torsion-free isometry
group $\G$ of $X$ satisfy one of the conditions \textup{(1), (2), (3)\,} above?
\end{quest}

The answer is yes if $X$ is visibility (i.e. any two points of $X(\infty)$
are endpoints of a geodesic, which happens e.g. if $K\le -1$), 
see Corollary~\ref{cor: visib horoball}.

Without any assumption on $X$ the answer is no, e.g. when $X/\G$ is
the product of two finite volume, complete, open, hyperbolic surfaces.

\begin{rmk} 
(a) The classes (1), (2), (3) are clearly not disjoint, e.g.
a the cyclic group generated by a translation in $\mathbb R^n$
lies in (1) and (3). On the other hand, any (discrete) subgroup 
satisfying (2) and (3) is virtually cyclic, see 
Proposition~\ref{prop: fixed pt discrete rk1 implies virt-Z}.\vspace{3pt}
\newline (b) 
A given group may be isomorphic to three different isometry groups 
satisfying (1), (2), (3) respectively.\vspace{3pt}
\newline (c)
If $\G\le\Iso(X)$ is a discrete subgroup, then 
the requirement ``$\G$ is isomorphic to a discrete isometry group 
of a Hadamard manifold that satisfies (3)''
does not restrict $\G$ because the isometric $\G$-action
on the warped product Hadamard manifold $\mathbb R\times_{e^r}\! X$
stabilizes a horoball. On the other hand, if we fix the
dimension, this does become a nontrivial restriction as follows.\vspace{3pt}
\newline (d) 
Define the {\it Euclidean action dimension\,}
as the smallest dimension of a Euclidean space
on which $\G$ acts smoothly and properly discontinuously,
and if the Euclidean action dimension of $\G$ equals $\dim(X)$,
then $\G$-action on $X$ cannot satisfy (3) because $L_\xi$ is
diffeomorphic to the Euclidean space of lower dimension.
See~\cite{BKK02, BesFei02, Yoo04, Des06} for computations
of a related invariant called the {\it action dimension\,} which
usually equals the Euclidean action dimension.
\end{rmk}

\section{Groups of non-parabolic isometries}
\label{sec: non-parabolic}

In this section we discuss groups in the class (1) of 
Section~\ref{sec: rough}.
An isometry of $X$  is {\it elliptic, axial\,}, or 
{\it parabolic\,} if the minimum of its displacement function is 
zero, positive, or not attained, respectively. 
If $\g$ is a non-parabolic isometry, then the set $\mathrm{Min}(\g)$
of points where the displacement functions attains a minimum 
splits as $C_\g\times \mathbb R^k$,
where $C_\g$ is a closed convex subset with
$\g$ acting as the product of the trivial action on $C_\g$
and a translation on $\mathbb R^k$~\cite[Theorem II.7.1]{BriHae}. 
If $\g$ is axial, its axes are precisely the lines $\{x\}\times\mathbb R$, $x\in C_\g$.

Theorems~\ref{thm: abelian semisimple}, \ref{thm: normalizer of abelian semisimple},
\ref{thm: semisimple nonembed}(2ab) are part of the 
flat torus theorem ``package'' discovered 
by Gromoll-Wolf~\cite{GroWol-flat-tor} and Lawson-Yau~\cite{LawYau-flat-tor}, 
and generalized in~\cite{BriHae}. 

\begin{thm}
\label{thm: abelian semisimple}
Let $A\le\Iso(X)$ be an abelian discrete subgroup 
that consists of non-parabolic isometries. Set $\mathrm{Min}(A):=\cap_{a\in A}\mathrm{Min}(a)$.
Then 
\newline 
\textup{(1)} $\mathrm{Min}(A)$ is a nonempty,
closed, convex $A$-invariant subset that splits 
as \newline\phantom{\textup{(1)}}  $C_A\times\mathbb R^m$ where $A$ acts trivially on $C_A$
and by translations on $\mathbb R^m$;\newline
\textup{(2)} $A$ is finitely generated of rank $\le m\le\dim(X)$;\newline
\textup{(3)}
$A$ has finite intersection with each conjugacy class in $\mathrm{Iso}(X)$.
\end{thm}

\begin{rmk}
Any periodic abelian discrete subgroup of $\Iso(X)$ is finite
(because it is countable, hence locally finite, and so is a union
of finite subgroups whose fixes points set is a descending family
of totally geodesics submanifolds of $X$, which has to stabilize by
dimension reasons).
Thus abelian discrete groups of 
non-parabolic isometries are finitely generated, which is 
a key feature of the Riemannian setting.
By contrast $\mathbb Q$ can act properly
by non-parabolic isometries on a proper CAT($0$) space, which
is the product of a simplicial tree and 
a line~\cite[Example II.7.13]{BriHae}. 
\end{rmk}

\begin{thm}
\label{thm: normalizer of abelian semisimple}
Let $A\le\Iso(X)$ be an abelian discrete subgroup 
that consists of axial isometries, and let
$N\le\Iso(X)$ is a subgroup that normalizes $A$. Then\newline
\textup{(4)} $N$ stabilizes 
$\mathrm{Min}(A)$ 
and preserves its product decomposition;\newline
\textup{(5)} $A$ is centralized by a finite index subgroup of $N$;\newline
\textup{(6)} $A$ is a virtual direct factor of $N$ if $N$ is finitely generated and $A\le N$.
\end{thm}
\begin{proof}[Proof of Theorems~\ref{thm: abelian semisimple},~\ref{thm: normalizer of abelian semisimple}]
(1) is proved in~\cite[Lemma 7.1(1)]{BGS}. Discreteness of $A$ and (1)
implies that $A$-action on $\mathbb R^m$ is properly discontinuous,
so $A$ has rank $\le m$, which gives (2).
Claim (3) follows from~\cite[Lemma II.7.17(2)]{BriHae}, 
and (4), (5), (6) are parts of the flat torus 
theorem~\cite[Theorem II.7.1]{BriHae}. 
\end{proof}

\begin{thm} 
\label{thm: semisimple nonembed}
Let $\G\le\mathrm{Iso}(X)$ be a subgroup without parabolic elements.
\newline
\textup{(1)} If $H$ is commensurable to $\G$, then $H$ is isomorphic
to a discrete group of \phantom{\textup{(3)}} non-parabolic isometries of some Hadamard manifold.
\newline
\textup{(2)} The following groups do not embed into $\G$:\newline
\phantom{\textup{(2)}} \textup{(a)}
 any solvable group that is not virtually abelian;
\newline
\phantom{\textup{(2)}} \textup{(b)}
the Baumslag-Solitar group $\langle x,y\,|\, xy^mx^{-1}=y^l\rangle$ with $m\neq\pm l$;
\newline
\phantom{\textup{(2)}} \textup{(c)}\! $\pi_1(L)$, where $L$ is a closed aspherical $3$-manifold
that admits no metric \newline \phantom{\textup{(2)} \textup{(c)}}
of $K\le 0$. 
\end{thm}
\begin{proof}
(1) is immediate
via the induced representation construction~\cite[Theorem 2.3]{KapLee-hadgr}. 
To prove (2a) note that $G$ must be polycyclic (combine finite generation of 
abelian discrete subgroups of non-parabolic isometries with Mal'cev's theorem that
a solvable group whose abelian subgroups are finitely generated is polycyclic).
Then proceed by induction on Hirsch length, and use  
Theorem~\ref{thm: normalizer of abelian semisimple}(6) to split virtual
$\mathbb Z$-factors one at a time, see~\cite[Theorem II.7.16]{BriHae}. 
Also (2b) follows from Theorem~\ref{thm: abelian semisimple}(3),
see~\cite[Theorem III.$\Gamma$.1.1(iii)]{BriHae}. Finally, to prove
(2c) invoke the solution of the virtual Haken conjecture~\cite{Ago-virthak}; thus
$L$ has a Haken finite cover,
so it admits a metric of $K\le 0$ by the main result in~\cite[Corollaries 2.6-2.7]{KapLee-hadgr}.
\end{proof}

\begin{rmk}\ \newline\vspace{-15pt}
\begin{enumerate}  
\item
If a closed aspherical $3$-manifold admits no metric of $K\le 0$,
then it is Seifert or graph (the manifold is virtually Haken~\cite{Ago-virthak}
so the claim follows from~\cite{Lee-3mnf-npc, KapLee-hadgr}).  
\item
Closed aspherical Seifert manifolds that admit no metric of $K\le 0$
are precisely those modelled on Nil, Sol, or $\widetilde{SL}_2(\mathbb R)$
as easily follows from Theorems~\ref{thm: abelian semisimple}(6),~\ref{thm: semisimple nonembed}(2a), 
and the observation that
the Seifert manifolds modelled on $\mathbb R^3$ or $\mathbb H^2\times\mathbb R$
are nonpositively curved.
\item
The problem which orientable closed graph manifolds
admit metrics of $K\le 0$ was resolved 
in~\cite{BuyKib-geometr-graph-II, BuySve-surv} 
who found several combinatorial criteria on the gluing data.
(These papers only consider manifolds with no embedded
Klein bottles, but the assumption can be easily removed as 
was explained to the author by Buyalo).
\item
A non-orientable closed $3$-manifold admits
a metric of $K\le 0$ if and only if its orientation cover 
does~\cite{KapLee-hadgr}.
\item A closed aspherical graph manifold admits a metric of $K\le 0$
if and only if its fundamental group virtually embeds into a 
right angled Artin group~\cite{Liu-graph}.
\item
Kapovich-Leeb used Theorem~\ref{thm: semisimple nonembed}(2c) to give 
other examples of groups that do not act
on Hadamard spaces by non-parabolic isometries~\cite{KapLee-hadgr}.
\end{enumerate}
\end{rmk}

\section{Groups with rank one elements: prelude}
\label{sec: rk1}

Sections~\ref{sec: rk1}--\ref{sec: orbit equiv}
collect what is known on the class (2) of 
Section~\ref{sec: rough}.

An axial isometry $\g$ has {\it rank one\,} if it has an axis 
that does not bound a flat half plane. 
This property can be characterized in terms of the
splitting $\mathrm{Min}(\g)\cong C_\g\times\mathbb R^k$, namely,
$\g$ has rank one if and only if $C_\g$ is compact.

Any discrete subgroup of $\Iso(X)$ that normalizes a rank one element
is virtually cyclic (in fact, the normalizer preserves
the splitting, and hence fixes a point of $C_\g$
and stabilizing the corresponding axis).

Rank one isometries were introduced by Ballmann~\cite[Theorem III.3.4]{Bal-book}, 
who proved that if $X$ is a rank one and $\G\le\Iso(X)$ 
is any subgroup satisfying the duality condition (e.g. a lattice), then $\G$ 
contains a rank one element. He then used a ping pong argument 
to find a copy of non-cyclic free group inside $\G$.
Various aspects of isometry groups containing rank one elements were further
studied in~\cite{BesFuj-geomtop2002, BalBuy-rk1,  
BesFuj-gafa2009, Ham-rk1, Ham-rk1-tot-disc, CapFuj}.
In particular, the following is due 
to~\cite[Proposition 5.11]{BesFuj-gafa2009} or~\cite[Theorem 1.1(4)]{Ham-rk1}:

\begin{thm}
\label{thm: rk1 element exists}
If $\G\le\mathrm{Iso}(X)$ is a discrete subgroup that contains
a rank one element and is not virtually-$\mathbb Z$, then $\G$ 
contains a non-cyclic free subgroup consisting of rank one elements.
\end{thm}

Sisto proved~\cite[Theorem 1.4]{Sis-contr-rand} that if $\G$
in Theorem~\ref{thm: rk1 element exists} is finitely generated,
then its generic element has rank one, where ``generic''
roughly means that the probability that a word written
in random finite generating set of $G$ represents a rank one element
approaches $1$ exponentially with the length of the word.

\section{Acylindrically hyperbolic groups and rank one elements}
\label{sec: acyl hyp}

In the last decade it was realized that many groups of geometric origin contain
(suitably defined) rank one elements, which allowed for a uniform treatment of such groups
and resulted in a host of applications. 
A crucial notion in these developments is acylindricity which goes back to 
Sela and Bowditch. In connection with rank one elements
different versions of acylindricity were introduced and studied by 
Bestvina-Fujiwara~\cite{BesFuj-geomtop2002, BesFuj-gafa2009},
Hamenst{\"a}dt~\cite{Ham-iso-gps}, Dahmani-Guirardel-Osin~\cite{DGO},
Sisto~\cite{Sis-contr-rand}, and most recently Osin showed~\cite{Osi-acy-hyp}
that all these approaches are equivalent. 

An isometric action of a group $G$ on a Gromov hyperbolic space $(X, d)$
is called
\begin{itemize}
\item {\it non-elementary\,}
if its limit set consists of $>2$ points, 
\vspace{1pt}\item  {\it acylindrical\,} 
if for each $\e>0$ there are $R, N$
such that if $d(x, y)\ge R$, then at most $N$ elements $g\in G$ 
satisfy $d(x, gx)\le\e$ and $d(y, gy)\le \e$.
\end{itemize}

A group $G$ is {\it acylindically hyperbolic\,}
if it admits a non-elementary acylindrical isometric action on
a Gromov hyperbolic space. 

The class of acylindically hyperbolic groups includes many
groups of geometric origin, e.g. any subgroup of 
a relatively hyperbolic group that is not virtually-cyclic
and does not lie is peripheral subgroup, or
all but finitely many mapping class groups; see~\cite{Osi-acy-hyp},
for other examples.

Of particular importance for this section is the following result of
Sisto~\cite{Sis-contr-rand} who actually 
proves it for any group acting properly and isometrically 
on a proper CAT($0$) space:

\begin{thm} 
\label{thm: sisto rk1}
{\bf (Sisto)}
If $\G\le\Iso(X)$ is a discrete subgroup that contains a rank one element, 
then $\G$ is virtually cyclic or acylindrically hyperbolic.
\end{thm}

Dahmani-Guirardel-Osin~\cite{DGO} introduced a notion of a
{\it hyperbolically embedded subgroup of $G$}, and Osin~\cite{Osi-acy-hyp} proved
that $G$ is acylindrically hyperbolic if and only if $G$ contains an infinite, proper,
hyperbolically embedded subgroup. What Sisto actually showed
is that any rank one element $\g\in \G$ lies in a virtually cyclic hyperbolically embedded
subgroup $E(\g)$. 
We omit the definition of a hyperbolically embedded subgroup, and just note that
they are almost malnormal by~\cite[Proposition 4.33]{DGO}:

\begin{thm}
If $H$ is a hyperbolically embedded subgroup of a group $G$, then
$H$ is almost malnormal in $G$, i.e. 
$H\cap gHg^{-1}$ is finite for all $g\notin H$.
\end{thm}

Here are other applications~\cite{DGO, Osi-acy-hyp}, which hold in particular when $G$
is a discrete, non-virtually-cyclic subgroup of $\Iso(X)$ containing a rank one element.

\begin{thm}
\label{thm: acy hyp list}
If $G$ is acylindrically hyperbolic, then\newline
\textup{(1)} $G$ has a non-cyclic, normal, free subgroup,
\newline
\textup{(2)} every countable group embeds into a quotient of $G$,
\newline
\textup{(3)} 
every infinite subnormal subgroup of $G$ is acylindrically hyperbolic,
\newline
\textup{(4)} $G$ has no nontrivial finite normal subgroups if and only if
every conjugacy \phantom{\textup{(6)}} class in $G$ is infinite,
\newline
\textup{(5)} every s-normal subgroup of $G$ is acylindrically hyperbolic,
\newline
\textup{(6)} if $G$ equals the product of subgroups $G_1,\dots, G_k$,
then at least one $G_i$ is\newline\phantom{\textup{(6)}} acylindrically hyperbolic.
\newline
\textup{(7)} $G$ is not the direct product of infinite groups. 
\newline
\textup{(8)} any group commensurable to $G$ is acylindrically hyperbolic.
\newline
\textup{(9)} any co-amenable subgroup of $G$ acylindrically hyperbolic.
\end{thm}
\begin{proof} Proofs of (1)--(4) are 
in~\cite[Theorem 8.6, Lemma 8.11, Theorem 8.12]{DGO} (1)--(2), while (5)--(7)
are proved in~\cite[Corollary 1.5, Proposition 1.7, Corollary 7.3]{Osi-acy-hyp},
and (8) appears in~\cite[Lemma 3.8]{MinOsi-acy--hyp}. 

The claim (9) is due to Osin who communicated to the author the following
argument and kindly permitted to include it here. 
Suppose $G$ is acylindrically hyperbolic and $K\le G$ is not.
In the next paragraph we find a non-cyclic free subgroup $F\le G$ with 
$F\cap K=\{ 1\}$. 
The group $F$ is not amenable, and its action by left translations on the set $G/H$ is free,
so there is no $F$-invariant finitely additive probability measure on $G/K$. 
The contrapositive of (9) now follows because
if $K$ were co-amenable,
$G/K$ would admit a $G$-invariant (and hence $F$-invariant) 
finitely additive probability measure.

To construct $F$ note that \cite[Theorems 1.1-1.2]{Osi-acy-hyp}
yield a non-elementary acylindrical $G$-action on a Gromov hyperbolic space
for which the subgroup $K$ is either elliptic or else virtually-$\mathbb Z$ and
contains a loxodromic element (we refer to~\cite{Osi-acy-hyp} for terminology). 
If $K$ is elliptic, then \cite[Theorems 1.1]{Osi-acy-hyp}
yields independent loxodromic elements $a,b\in G$.
The standard ping-pong argument shows that for some $n\gg 1$, the
subgroup $\langle a^n,b^n\rangle$ is free of rank $2$ and all 
its non-trivial elements are loxodromic. In particular,  
$\langle a^n,b^n\rangle\cap K=\{ 1\}$. If $K$ is
virtually-$\mathbb Z$ and contains a loxodromic element $c$, then \cite[Theorems 1.1]{Osi-acy-hyp}
gives loxodromic elements $a,b\in G$ such that $a,b,c$ are independent.
Again by ping-pong  $\langle a^n, b^n, c^n\rangle $ is free of rank $3$ 
so that $\langle a^n, b^n\rangle\cap K=\{ 1\}$. Thus we get 
a non-cyclic free subgroup $F=\langle a^n,b^n\rangle$ with $F\cap K=\{ 1\}$, and
in fact all nontrivial elements of $F$ are loxodromic.
\end{proof}

A subgroup $K\le G$ is {\it subnormal\,} if there are 
subgroups $G_i\le G$  with $G_0=G$, $G_k=K$, 
and such that $G_{i}$ is a normal in $G_{i-1}$ for all $i=1,\dots , k$.

If a group $G$ equals the product of subgroups $G_1,\dots, G_k$,
one says that $G$ {\it boundedly generated by} $G_1,\dots, G_k$.

Two groups are {\it commensurable\,}
if they have isomorphic finite index subgroups.

A subgroup $K\le G$ is {\it s-normal\,} if $K\cap gKg^{-1}$
is infinite for each $g\in G$. Thus 
the Baumslag-Solitar group $B(m,n)=\langle a,b\,|\, ab^m=b^na\rangle$
is not acylindrically hyperbolic, except for $B(0,0)$, because
$\langle b\rangle$ is s-normal and not acylindrically hyperbolic. 

\label{sec: co-amen}
A subgroup $K\le G$ is {\it co-amenable} if one of the following holds:
\begin{enumerate}
\item
every continuous affine 
$G$-action on a convex compact subset of a locally convex space with a
$K$-fixed point has a $G$-fixed point;
\item
$\ell^\infty(G/K)$ has a $G$-invariant mean;
\item
$G/K$ has a $G$-invariant finitely additive probability measure;
\item  
the inclusion $K\hookrightarrow G$ induces injections in bounded cohomology
in all degrees with coefficients in any dual Banach 
$G$-module.
\end{enumerate} 
The equivalence $\text{(1)}\Leftrightarrow \text{(2)}$ is proved in~\cite{Eym-book},
while $\text{(2)}\Leftrightarrow \text{(3)}$ follows from the standard correspondence
between means and measures, and $\text{(3)}\Leftrightarrow \text{(4)}$ can be found
in~\cite{MonPop-co-amen}. See also~\cite{Moo-thesis} for leisurely discussion of co-amenability.

Here we are mainly
interested in examples of non-amenable groups that admit
co-amenable subgroups:  
\begin{itemize}
\item
A normal subgroup $N\unlhd\, G$ is co-amenable if and only if $G/N$ is amenable.
\item
If $K$ is co-amenable in $N$, and in turn $N$ is co-amenable in $G$,
then $K$ is co-amenable in $G$.
\item
the image of a co-amenable subgroup under an epimorphism $G\to \bar G$
is co-amenable.
\item
If $\theta\co K\to K$ is a monomorphism, and $G:=\langle K,t\,|\, tkt^{-1}=
\theta(k), k\in K\rangle$ is the associated HNN-extension, then $K$ is co-amenable
in $G$.
\end{itemize}
The first three facts above are straightforward, while the last one is due to 
Monod-Popa~\cite{MonPop-co-amen}. 

\begin{ex}
Starting from a group 
$K$ that is not acylindrically hyperbolic one can 
use iterated HNN-extensions and extensions 
with amenable quotient 
to get many examples of non acylindrically hyperbolic groups.
(These constructions preserve finiteness of cohomological dimension
if the initial $K$ and every amenable quotient
have finite cohomological dimension).
\end{ex}

\section{Bounded cohomology and rank one elements}

Bounded cohomology naturally appear in a variety of 
contexts, see e.g.~\cite{Gro-vol-bd-coh, Mon-icm}. Of particular interest
for our purposes is the {\it comparison map\,} 
\[
\iota (G)\co H^2_b(G;\mathbb R)\to H^2(G;\mathbb R)
\]
between the bounded and ordinary cohomology in degree two, which encodes
some subtle group-theoretic properties: \vspace{3pt}\newline
$\bullet$
\bf (Johnson)\rm\ If $G$ is amenable, then $H^p(G;\mathbb R)=0$ for $p>0$~\cite{Jon-bound-coh, Nos-bound-coh}.
\vspace{3pt}\newline
$\bullet$ \bf (Burger-Monod)\rm\ 
$\iota(G)$ is injective if $G$ is the fundamental group of an irreducible, finite volume
complete manifold of $K\le 0$, no local Euclidean de Rham factor, 
and rank $\ge 2$~\cite{BurMon}.
\vspace{3pt}\newline 
$\bullet$ \bf (Bavard)\rm\ Injectivity of $\iota(G)$ is equivalent to
vanishing of the stable commutator length
on $[G,G]$~\cite{Bav-commut}. Thus if $\iota(G)$ is non-injective, then
there is $g\in [G,G]$ such that the minimal number of commutators
needed to represent $g^n$ grows linearly with $n$.

\begin{thm} \label{thm: bestv-fuj-rk1}
\bf (Bestvina-Fujiwara, Osin)\rm\
If $G$ is acylindrically hyperbolic, then the comparison map 
$\iota(\G)$ has infinite dimensional kernel.
\end{thm}

This was proved in~\cite{BesFuj-geomtop2002} for a class of
groups which according to~\cite{Osi-acy-hyp} coincides with the class of
acylindrically hyperbolic groups.

\begin{rmk} 
For discrete subgroup $\G\le\Iso(X)$ with rank one elements
the above theorem was first established in~\cite{BesFuj-gafa2009}.
Discreteness of $\G$ in Theorem~\ref{thm: bestv-fuj-rk1} can be weakened to
the weak proper discontinuity~\cite{BesFuj-gafa2009}, 
but it cannot be dropped, e.g. the projection of any irreducible lattice  
$\Lambda\le\Iso({\mathbf H}^2)\times \Iso({\mathbf H}^2)$ to either factor
acts on the hyperbolic plane isometrically, effectively, and by
rank one isometries, but the comparison map $\iota(\Lambda)$
is injective~\cite{BurMon}.
\end{rmk}

\section{Monod-Shalom's class and rank one elements}
\label{sec: orbit equiv}

In~\cite{MonSha-jdg2004, MonSha-orb-ann2006},
Monod-Shalom introduced and studied the following class of groups,
which they thought of as a 
cohomological manifestation of negative 
curvature: Let $\mathcal C_{\text{reg}}$ be the class of countable
groups $G$ such that $H^2_b(G; \ell^2(G))\neq 0$, which
refers to the bounded cohomology of $G$ with coefficients in the
regular representation.

A way to prove that $H^2_b(G; \ell^2(G))\neq 0$ is to show that
the corresponding comparison map $H^2_b(G; \ell^2(G))\to H^2(G; \ell^2(G))$
has infinite dimensional kernel, which was done for many
``hyperbolic-like'' groups in~\cite{MinMonSha}. The following was proved by 
Hamenst{\"a}dt~\cite{Ham-bound-coh}, and later from
a different perspective by Hull-Osin~\cite{HulOsi}:

\begin{thm}  
$\mathcal C_{\text{reg}}$ contains every countable
acylindrically hyperbolic group, and hence any discrete subgroup $\G\le\Iso(X)$
that contains a rank one element and is not virtually-$\mathbb Z$.
\end{thm}

For discrete subgroup $\G\le\Iso(X)$ with rank one elements
the above theorem was first established in~\cite{Ham-rk1-tot-disc}.

As proved in~\cite[Chapter 7]{MonSha-orb-ann2006}, 
examples of groups not in $\mathcal C_{\text{reg}}$
include 
\begin{itemize}
\item amenable groups, 
\item products of at least two infinite
groups, 
\item lattices in higher-rank simple Lie groups (over any local field), 
\item irreducible lattices in products of compactly generated non-amenable
groups.
\end{itemize}
and the class $\mathcal C_{\text{reg}}$ is closed under 
\begin{itemize}
\item
passing to an infinite normal subgroup, 
\item passing to a co-amenable subgroup,
\item measure equivalence.
\end{itemize}
We refer to~\cite{Fur-surv} for a survey on measure equivalence;
e.g. commensurable groups are measure equivalent.

\begin{quest}
Is every group in $\mathcal C_{\text{reg}}$ acylindrically hyperbolic?
\end{quest}

To transition to our next topic, note that  
non-virtually-cyclic discrete groups with rank one elements
never fix a point at infinity:

\begin{prop}
\label{prop: fixed pt discrete rk1 implies virt-Z}
If $\G$ is discrete, fixes a point at infinity, and contains a rank one element, then $\G$
is virtually cyclic. 
\end{prop}
\begin{proof}
This follows from~\cite[Section 6]{BesFuj-gafa2009} provided $\G$ satisfies the weak proper
discontinuity condition, which is implied by discreteness. The idea is that if $g\in\G$ has rank
one, then either $\G$ is virtually-$\mathbb Z$,
or $\G$ contains another rank one element $h$ such that their axis $A_g$, $A_h$
do not have the same sets of endpoints at infinity. A rank one element fixes precisely
two points at infinity, the endpoints of its axis. Since $\G$ has a fixed point, it
must be a common endpoint of $A_g$, $A_h$, and this contradicts
weak proper discontinuity: there is a subsegment $I$
of $A_g$ and an infinite subset $Q$ of $\G$ such that 
the distances between the endpoints of $I$ and $g(I)$, $g\in Q$ are uniformly bounded. 
\end{proof}

\begin{ex}
The (non-discrete) stabilizer of a boundary
point in the hyperbolic plane contains rank one
elements without being virtually-$\mathbb Z$.
\end{ex}

\section{Groups that fix a point at infinity: prelude}
\label{sec: fix a point at infinity}

Basic properties of horoballs, horospheres, and Busemann functions
can be found in~\cite{BGS, Ebe-book, Bal-book}.
A {\it horoball} in $X$ is the Hausdorff limit of a sequence of 
metric balls in $X$ with radii going to infinity. 
A {\it horosphere\,} is the boundary of a horoball. 
Every point at infinity $\xi$ is represented by a 
Busemann function $b_\xi\co X\to \mathbb R$, which is
determined by $\xi$ up to an additive constant. 
The fibers of $b_\xi$ are the horospheres centered at
$\xi$, and the sublevel sets of $b$ are horoballs
centered at $\xi$. 
The function
$b_\xi$ is a $C^2$ Riemannian submersion $X\to\mathbb R$,
and in particular, each horosphere is diffeomorphic to 
the Euclidean space of dimension $\dim(X)-1$.

If $\G$ fixes a point $\xi$ at infinity of $X$, then $\G$
permutes horospheres centered at $\xi$, and 
associating to $\g\in\G$ the distance  by which 
it moves a horosphere to a concentric one
defines a homomorphism $\G\to\mathbb R$,
which is in general nontrivial (think of the
stabilizer of a point at infinity of the hyperbolic plane).

Let $L_\xi$ be the space of lines in $X$ asymptotic to $\xi$.
The geodesic flow towards $\xi$ identifies $X$ with the total space of 
a principal $\mathbb R$-bundle
over $L_\xi$, which is trivial as every horosphere
centered at $\xi$ gives a section. In particular, 
$L_\xi$ has a structure of a smooth manifold diffeomorphic
to a horosphere about $\xi$. 
If $\G$ fixes $\xi$, then it acts smoothly on $L_\xi$. 

Recall the condition (3) of Section~\ref{sec: rough}:
$\G$ fixes a point $\xi\in X(\infty)$, and the associated $\G$-action
on  the space $L_\xi$ of lines asymptotic to $\xi$ is
free and properly discontinuous. Under this condition the principal $\mathbb R$-bundle
$X\to L_\xi$ descends to an orientable (and hence trivial) real line bundle
$X/\G\to L_\xi/\G$, so we get:

\begin{lem}
If $\G$ satisfies the condition \textup{(3)} of \textup{Section~\ref{sec: rough}},
then $X/\G$ is diffeomorphic to the product of $\mathbb R$ and $L_\xi/\G$.
\end{lem}

The prime example of  a group satisfying (3)
is a discrete torsion-free subgroup $\G$ that stabilizes a horoball, 
in which case $\G$
stabilizes every concentric horoball, so that 
$X$ is $\Gamma$-equivariantly
is diffeomorphic to the product of $\mathbb R$ with a horosphere,
and (3) follows because $\G$ acts freely and properly discontinuously on $X$.
More examples are needed:

\begin{quest} Let $\G$ be any discrete torsion-free isometry group of $X$
whose fixed point set at infinity is nonempty.
\vspace{-4pt}
\begin{itemize}
\item
Does $\G$ satisfies \textup{(3)\,} for some $\xi$? 
\item
If $\G$ satisfies \textup{(3)\,}, does $\G$ stabilize a horoball?
\item
What is the structure of $\G$ if it does not stabilize a horoball?
\end{itemize}
\end{quest}

\begin{rmk} (1)
An axial isometry does not stabilize a horoball
centered at an endpoint of one of its axis, but it
can stabilize another horoball (e.g. translation 
in the plane stabilizes any half plane
whose boundary is parallel to the translation axis).
\newline (2)
Any parabolic isometry stabilizes a horoball by 
Lemma~\ref{lem: center has parabolic}
but different parabolics in $\G$ can stabilize
different horoballs.
\newline (3)
An elliptic element fixing a point at infinity
stabilizes a horoball centered at the point (because
it fixes a ray from a fixes point inside $X$ to the
fixed point at infinity).
\end{rmk}

\begin{quest}
\label{quest: no rank one implies grounded}
Let $\G\le\Iso(X)$ be discrete, containing a parabolic and no rank one elements.
What conditions on $X$ ensure that $\G$ fixes a point at infinity?  
\end{quest}

Recall that the limit set
$\Lambda (\G)$ is the set of accumulation points of
the $\G$-orbit of a point of $X$. 
Ballmann-Buyalo~\cite[Proposition 1.10]{BalBuy-rk1} gave the following
characterization of groups containing rank one elements
in terms of the Tits radius of the limits set: 

\begin{prop}\label{prop: ballmann-buyalo, rank 1}
$\G$ contains no rank one element if and only if
$\Lambda(\G)$ lies in the Tits ball of radius $\le\pi$ centered at 
a point of $\Lambda(\G)$.
\end{prop}

Combining this characterization with results of Schroeder~\cite[Appendix 3]{BGS}
one can answer Question~\ref{quest: no rank one implies grounded} 
when every component of the Tits boundary of $X$ has radius $\le\frac{\pi}{2}$:

\begin{cor}
\label{cor: tits exist rank one}
If every component of $X(\infty)$ equipped with the Tits metric has radius 
$\le\frac{\pi}{2}$,
then a subgroup $\G\le\mathrm{Iso}(X)$ either contains a rank one element or
fixes a point at infinity. 
\end{cor}
\begin{proof}
Following~\cite[page 220]{BGS}, for
a subset $Q\subset X(\infty)$,
let \[C_Q=\{z\in Q\,|\, Q\ \text{lies in the Tits ball of radius}\ \frac{\pi}{2}
\ \text{centered at}\ z\}.\]
Lower semicontinuity of the Tits distance implies
that if $Q$ is closed in the cone topology on $X(\infty)$, then so is
$C_Q$~\cite[4.9]{BGS}. If $C_Q$ is nonempty, then
clearly it has Tits diameter $\le\frac{\pi}{2}$.

Apply this to $Q=\Lambda(\G)$, which is a closed $\G$-invariant subset. 
Since $\G$ has no rank one element, Proposition~\ref{prop: ballmann-buyalo, rank 1}
implies that $\Lambda(\G)$ lies in the Tits ball of radius $\le\pi$
about one of its points,
and by our assumption the ball must have radius $\le\frac{\pi}{2}$
(for Tits metric is length so the distance between different components is infinite).
Thus $C_{\Lambda(\G)}$
is a closed subset of Tits diameter $\le\frac{\pi}{2}$.

By the main result of~\cite[Appendix 3]{BGS} 
any subset of $X(\infty)$ that is closed in the cone topology
and has Tits diameter $\le\frac{\pi}{2}$ has a unique {\it center}, 
defined as the center of the closed
(Tits) ball of the smallest radius among all
balls containing the subset. Let $z_G$ be the
unique center of $C_{\Lambda(\G)}$. Since $\Lambda(G)$
is  $\G$-invariant, so is $C_{\Lambda(\G)}$, and hence
$\G$ fixes $z_G$.
\end{proof}

\begin{ex}
The components of the Tits 
boundary are points if (and only if) $X$ is visibility~\cite[4.14]{BGS}, 
so Corollary~\ref{cor: tits exist rank one}
applies if $X$ is visibility, in which case one can say more:
\end{ex}

\begin{cor}
\label{cor: visib horoball}
If $X$ is visibility and $\G$ contains a parabolic element but no rank one elements, 
then $\G$ contains no axial isometries, $\Lambda (\G)$ is a point,
the fixed point set of $\G$ at infinity equals
$\Lambda (\G)$, and $\G$ stabilizes every horoball centered at $\Lambda (\G)$.
\end{cor}
\begin{proof}
Since $\G$ contains no rank one elements, 
Proposition~\ref{prop: ballmann-buyalo, rank 1}
implies that $\Lambda(\G)$ is a point
(as the Tits distance between any two distinct points is infinite).
Being a visibility space, $X$ has no flat half spaces, so $\G$ contains no axial isometries.
The limit set is $\G$-invariant, so $\Lambda(\G)$ is a fixed
point of $\G$. 
If $\G$ fixed any other point, it would
also be fixed by the cyclic subgroup generated by a parabolic 
in $\G$, but the fixed point set of any abelian subgroup
containing a parabolic
has Tits radius $\le\frac{\pi}{2}$, which again is a single 
point. Thus $\Lambda(\G)$ is a unique fixed point of $\G$.
\end{proof}

\begin{rmk} By Corollary~\ref{cor: visib horoball} and
Proposition~\ref{prop: fixed pt discrete rk1 implies virt-Z}
any non-virtually-cyclic, discrete isometry group of a visibility manifold
that fixes a point at infinity must stabilize a horoball.
There is a sizable class of groups to which this applies, e.g.
by Corollary~\ref{cor: commut subgr grounded} it contains the product
of any nontrivial torsion-free groups, see~\cite{KarNos04} 
for more examples.
\end{rmk}

\section{Groups whose center contains a parabolic} 
\label{sec: center with parabolic}

The following result is implicit in~\cite[Lemma 7.3, 7.8]{BGS}.

\begin{lem} 
\label{lem: center has parabolic}
$\G\le\Iso(X)$ stabilizes a horoball
if it has a finite index subgroup $\G_0$ whose center $Z(\G_0)$
contains a parabolic isometry.
\end{lem}
\begin{proof}
Fix a parabolic isometry $z\in Z(\G_0)$, 
and right coset representatives $g_1,\dots g_k$ 
of $\G_0$ in $\G$. Then the 
function $x\to \sum_i d(g_i^{-1}zg_i x, x)$
is convex and $\G$-invariant. 
Since $x\to d(zx, x)$ does not assume its infimum,
neither does the above convex function. Then
a limiting process outlined in~\cite[Lemma 3.9]{BGS},
cf.~\cite[Lemma II.8.26]{BriHae},
gives rise via Arzela-Ascoli theorem to a $\G$-invariant
Busemann function, whose sublevel sets are $G$-invariant horoballs. 
\end{proof}

Flat torus theorem~\cite[Chapter II.7]{BriHae} restricts 
a discrete isometry group of $X$ whose center consists of axial
elements, which can be summarized as follows, see~\cite{Bel-bus}.

\begin{thm}\label{thm: into-center}
Let $G$ be a group with subgroups $H$, $G_0$ such that
their centers $Z(H)$, $Z(G_0)$ are infinite, 
$Z(H)\subseteq Z(G_0)$, the index of $G_0$ in $G$
is finite, and one of the following conditions hold: \newline
\textup{(1)} $Z(H)$ is not finitely generated;\newline
\textup{(2)} any homomorphism $H\to\mathbb R$ is trivial.\newline
\textup{(3)} $H$ is finitely generated, and
$Z(H)$ contains a free abelian subgroup
that is \newline\phantom{\textup{(3)}} 
not a direct factor of any 
finite index subgroup of $H$.\newline
If a discrete subgroup of $\Iso(X)$ is isomorphic to $G$, then it 
stabilizes a horoball.
\end{thm}

The reader may want to first think through 
the case when $H=G_0=G$, and then go on
to observe that if Theorem~\ref{thm: into-center}
holds for $H$, $G_0$, $G$, then it also 
does for $H$, $G_0\times K$, $G\times K$
for any group $K$.

\begin{ex}
\label{ex: rationals}
Theorem~\ref{thm: into-center}(1) applies, 
e.g. to any infinitely generated, torsion-free, countable
abelian group of finite rank, such as $(\mathbb Q, +)$,
where finiteness of rank ensures finiteness of cohomological 
dimension~\cite[Theorem 7.10]{Bie-book}. 
\end{ex}

\begin{rmk}
It is unknown whether there is 
a group of type {\it F} with infinitely generated
center. Such a group  cannot be elementary amenable~\cite{Bel-bus},
or linear over a field of characteristic 
zero~\cite[Corollary 5]{AlpSha82}.
Sanity check: there does exist a finitely presented group 
with solvable word problem
whose center contains every countable abelian 
group~\cite[Corollary 3]{Oul-center}.
\end{rmk}

\begin{ex} 
\label{ex: with center (3)}
(groups of type {\it F\,} to which 
Theorem~\ref{thm: into-center}(3) applies, see~\cite{Bel-bus}):
\begin{enumerate}
\item
$H$ is the fundamental group of the total space of any
principal circle bundle with non-zero rational
Euler class and a finite aspherical cell complex
as the base.
\item
$H$ is a torsion-free, finitely generated, non-abelian 
nilpotent group. 
\item
\label{item: Seifert 3-manifold}
$H$ is the fundamental group of any 
closed orientable Seifert $3$-manifold modelled
on $\widetilde{SL}_2(\mathbb R)$.
\item 
\label{item: Sp lattice}
$H$ is the preimage of any torsion-free lattice in $Sp_{2n}(\mathbb R)$
under the universal 
cover $\widetilde{Sp}_{2n}(\mathbb R)\to Sp_{2n}(\mathbb R)$ for $n\ge 2$. 
\item
\label{item: amalgam over center}
$H$ is the amalgamated product $G_1\ast_A G_2$ where  
$G_1$, $G_2$ have type {\it F} and are finitely generated, 
$A$ lies in the center of $G_1$, $G_2$
and contains a subgroup that
is not a virtual direct factor of $G_1$.
\end{enumerate}
\end{ex}

\section{Anchored groups and fixed points at infinity}
\label{sec: anchored}

Let us discuss algebraic conditions that force
an isometry group of $X$ to fix a point at infinity. 

If a subgroup $\G\le\Iso(X)$ contains a parabolic element,
stabilizes a closed convex noncompact subset $W\subseteq X$,
and fixes a point at infinity of $W$,  
then we say that $\G$ {\it is anchored in\,} $W$. 
(Passing to an invariant closed convex subset is essential 
in some inductive arguments, e.g. in Theorem~\ref{thm: co-amen grounded}).

\begin{thm} 
\label{thm: anchored} Let $\G\le\Iso(X)$ be a subgroup and $W$
be a any closed, convex, noncompact $\G$-invariant subset of $X$.
Then $\G$ is anchored in $W$ if one of the following holds:\newline
\textup{(1)} $\G$ is abelian and contains a parabolic.
\newline
\textup{(2)} $\G$ has a normal subgroup that is anchored in $W$.
\end{thm}
\begin{proof} (1) For $W=X$ this is due to Schroeder~\cite[Appendix 3]{BGS},
see also cf.~\cite{FNS}. The general case follows from a result of 
Caprace-Lytchak~\cite[Corollary 1.5]{CapLyt10} 
that the centralizer of a parabolic isometry
of a CAT($0$) space of finite telescopic dimension
has a fixed point at infinity.
Closed convex subsets of Hadamard manifolds have
finite telescopic dimension, see~\cite[Section 2.1]{CapLyt10},
and a parabolic isometry of $X$ that stabilizes a closed convex
subset acts in that subset as a parabolic isometry.

(2) For $W=X$ this is due to Eberlein~\cite[Proposition 4.4.4]{Ebe-book},
whose proof generalizes to our setting 
via~\cite[Proposition 5.7]{FNS} or~\cite[Proposition 1.4]{BalLyt}.
\end{proof}

Given a class of Hadamard manifolds $\mathcal C$,
we say that a group $G$ is {\it clinging in\,} $\mathcal C$ 
if for any discrete subgroup $\G\le\Iso(Y)$ such
that $Y\in \mathcal C$ and $\G$ is isomorphic to $G$, and for any
$\G$-invariant closed convex noncompact subset $W$ of $Y$, 
the group $\G$ is anchored in $W$. If $G$ is clinging in the class
of all Hadamard manifolds, we simply call $G$ {\it clinging}.

In particular, $G$ is clinging in $\mathcal C$ if no such $\G$ exists but we of course
are interested in nontrivial examples. 
Theorem~\ref{thm: clinging} and Corollary~\ref{cor: elem amen} below
can be found in~\cite{Bel-bus}.

\begin{thm} 
\label{thm: clinging}
A group $G$ is clinging in $\mathcal C$ if one of the following is true:
\newline
\textup{(1)}
$G$ has a clinging in $\mathcal C$ normal subgroup, or
\newline
\textup{(2)}
$G$ is the union of a nested sequence of clinging in $\mathcal C$ subgroups.
\newline
\textup{(3)}
$G$ is as in \textup{Theorem~\ref{thm: into-center}}.
\newline
\textup{(4)}
$G$ is virtually solvable and not virtually-$\mathbb Z^k$ for any $k$.
\newline
\textup{(5)}
a normal abelian subgroup of $G$ 
contains an infinite $G$-conjugacy class.
\end{thm}

\begin{cor}
\label{cor: elem amen}
Let $G$ be a finitely generated,
torsion-free group that has a nontrivial, normal, elementary amenable 
subgroup. Then either $G$ is clinging, or $G$ has a nontrivial,
finitely generated, abelian, normal subgroup that is a  
virtual direct factor of $G$.
\end{cor}

Splitting results of Schroeder~\cite{Sch-split-1985} and Monod~\cite{Monod-superrid-jams} give another
source of groups fixing points at infinity.

\begin{thm}
\label{thm: commut subgr grounded}
Let $W$ be a closed, convex, noncompact subset of a Hadamard manifold. 
If a discrete torsion-free isometry group of $W$
contains two commuting subgroups $\G_1$, $\G_2$, 
then one of them fixes a point at infinity of $W$.
\end{thm}
\begin{proof}
Suppose neither $\G_1$ nor $\G_2$ fixes a point at infinity of $W$,
and in particular $\G_1\G_2$ is nontrivial.
Since $\G_1$ fixes no point at infinity, 
it follows from~\cite[Proposition 27, Remark 39 and Subsection 4.6]{Monod-superrid-jams}
that $W$ contains a non-empty
minimal convex closed $\G_1$-invariant subset $C_1$, and moreover,
the union $C$ of such sets splits as $C_1 \times C_2$
for some bounded convex subset $C_2$ where $\G_1$, $\G_2$
preserve the splitting and act trivially on $C_2$, $C_1$,
respectively. 
(Boundedness of $C_2$ is a key point, so we
explain it here: if $C_2$ is unbounded, it contains a ray 
$s\to r(s)$, so given $x\in C_1$ we get a ray $\{x\}\times r$ in $X$
which is mapped by any $\g\in \G_1$ to an asymptotic ray as
$\g$ maps $(x,r(s))$ to $(\g(x), r(s))$; thus $\G_1$ fixes a point
at infinity contradicting the assumptions).
Since $C_2$ is bounded, $\G_2$ fixes the circumcenter of $C_2$,
and hence fixes a point $z_2\in C$. Repeating
the same argument with $\G_2$, $\{z_2\}$ in place of $\G_1$, $C_1$
shows that the union $Z$ of minimal convex closed $\G_2$-invariant subsets
splits as the product of $\{z_2\}$ and a bounded convex subset,
and the splitting is invariant under $\G_1$, $\G_2$.
It follows that $\G_1\G_2$ fixes a point of $Z$.
This is where we need that $\G_1\G_2$ lies in a torsion-free discrete subgroup,
because it implies that $\G_1\G_2$ is trivial.  
\end{proof}

\begin{cor}
\label{cor: clinging no-parabolic}
If $G_1$ and $G_2$ are groups each containing a subgroup
as in \textup{Theorem~\ref{thm: semisimple nonembed}(2)},
then $G_1\times G_2$ is clinging in the class of all
Hadamard manifolds.
\end{cor}
\begin{proof}
Suppose $G_1\times G_2$ is isomorphic to a discrete subgroup $\G\le\Iso(X)$
stabilizing a closed, convex, noncompact subset $W$.
By Theorem~\ref{thm: semisimple nonembed}(2),
each factor contains a parabolic, and one of them fixes a point at infinity by
Corollary~\ref{thm: commut subgr grounded}, and hence is anchored in $W$.
So $\G$ is anchored in $W$ by Theorem~\ref{thm: anchored}(2).
\end{proof}

\begin{cor}
\label{cor: commut subgr grounded}
If $G_1$, $G_2$ are nontrivial torsion-free groups, then
$G_1\times G_1$ is clinging in the class of Hadamard manifolds
containing no flat half planes.
\end{cor}
\begin{proof} 
Suppose $G_1\times G_2$ is realized as a discrete isometry group
of a Hadamard manifold with no flat half planes, and suppose
$W$ is a closed, convex, noncompact invariant subset.
Since $G_1\times G_2$  is torsion-free, by
Theorem~\ref{thm: commut subgr grounded} one of the factors
fixes a point at infinity of $W$.
If say $G_1$ contains a hyperbolic element $h$,
then $h$ has rank one as $X$ contains no flat half planes, so
the centralizer of $h$ in $G_1G_2$ is cyclic, and also contains $h$ and $G_2$
violating the assumption that $G_2$ is nontrivial.
Thus $G_1$ consists of parabolics, and by symmetry so does $G_2$.
One of the groups $G_1$, $G_2$ fixes a point at infinity, hence it is
anchored in $W$, and so is $G_1G_2$ by Theorem~\ref{thm: anchored}(2). 
\end{proof}

Results of Caprace-Monod imply:

\begin{thm}
\label{thm: co-amen grounded}
If $H$ is clinging in a class of Hadamard manifolds $\mathcal C$,
and $G$ contains $H$ as a co-amenable subgroup, then $G$ is clinging in $\mathcal C$.
\end{thm}
\begin{proof}
Realize $G$ as a discrete isometry group of a Hadamard manifold in $\mathcal C$
stabilizing a closed, convex, noncompact subset $W$.
By assumption $H$ contains a parabolic, hence so does $G$.
If $G$ does not fix a point at infinity of $W$,
then by~\cite[Theorem 4.3]{CapMon09} $W$ contains a
minimal closed convex $G$-invariant subspace $U$. 
Note that $U$ has no Euclidean de Rham factor
(as the other factor would then be a smaller $\G$-invariant subset).
The co-amenability implies that $H$ fixes no point at infinity of 
$U$~\cite[Proposition 2.1]{CapMon-discrete}
so it is not clinging in $\mathcal C$.
\end{proof}

\begin{rmk} 
Burger-Schroeder~\cite{BurSch87} showed that any amenable
subgroup $\G\le\Iso(X)$ either fixes a point at infinity 
or stabilizes a flat, and this generalizes to actions on proper
CAT($0$) spaces by Adams-Ballmann~\cite{AdaBal}.
Even more generally, 
Caprace-Monod~\cite[Corollary 2.2]{CapMon-discrete} 
obtained the same conclusion whenever
$\G$ contains two commuting co-amenable subgroups (and also
gave examples with non-amenable $\G$). 
To make this result into a source of groups
that fix a point at infinity more examples
are needed, and with our focus
on manifolds one has to answer the following.
\end{rmk}

\begin{quest}
Is there a group that contains
two commuting co-amenable subgroups, has finite cohomological 
dimension, and is not virtually solvable?
\end{quest}

As mentioned in Section~\ref{sec: fix a point at infinity}, 
if $\G$ fixes a point at infinity, then $\G$
permutes horospheres centered at the point
defining a homomorphism $\G\to\mathbb R$.
Thus {\it if $\G\le\Iso(X)$ has no nontrivial homomorphism
into $\mathbb R$, then
$\G$ stabilizes a horoball if and only if 
$\G$ fixes a point at infinity.} 

Examples of clinging groups 
with finite abelianization (and hence no nontrivial homomorphisms into $\mathbb R$)
are abound, see~\cite{Bel-bus}, and there are many such groups of 
finite cohomological dimension, or even of type {\it F}
(so they may well be the fundamental groups of complete manifolds of $K\le 0$).   

Note that the property of having finite abelianization
is inherited by amalgamated products (clearly), and by extensions
(due to right exactness of the abelianization functor). 
Moreover, an extension with a finite quotient often has 
finite abelianization, e.g. 
the abelianization of the semidirect product
$A\rtimes B$ is $(A^\ab)_B\times B^\ab$, where
$(A^\ab)_B$ is the coinvariants for the $B$-action 
on $A^\ab$.

\section{Homotopy obstructions (after Gromov and Izeki-Nayatani)}
\label{sec: homotopy obstr}

If $M$ is a complete manifold of $K\le 0$, then 
$\pi_1(M)$ has finite cohomological dimension.
A group has finite cohomological dimension if and only if 
it is the fundamental group
of a manifold whose universal cover is diffeomorphic to 
a Euclidean space.

Gromov asked~\cite{Gro-asy} whether every countable 
group of finite cohomological dimension
is isomorphic to some $\pi_1(M)$ where $M$ is complete of $K\le 0$. 
(The question is a good illustration of 
how little we know about open manifolds of $K\le 0$.)
The answer is no due to groundbreaking
works of Gromov~\cite{Gro-rand03} and Izeki-Nayatani~\cite{IseNay05} on groups
with strong fixed point properties. 

These papers combine certain averaging procedures with ideas of harmonic map 
superrigidity to produce many a group $G$ such that
\begin{itemize}
\item[(a)]
any isometric $G$-action on a Hadamard manifold  
has a fixed point;
\vspace{2pt}
\item[(b)]
$G$ has type {\it F} (i.e. is the fundamental group of a
compact aspherical manifold with boundary).
\end{itemize}
Since the $\pi_1(M)$-action on the universal cover of $M$
is free, it follows that there is no complete manifold $M$ of 
$K\le 0$ and $\pi_1(M)\cong G$. 
The methods actually reach far beyond Hadamard manifolds,
and apply to isometric $G$-actions on a wide variety of spaces,
see ~\cite{IseNay-surv, NaoSil11, IKN}.

Gromov's examples are certain torsion-free hyperbolic groups produced 
from a sequence of graphs $\G_n$ whose edges are labeled with words
of length $j$ in an alphabet of $d>1$ letters. The words
are chosen randomly, and reversing orientation of the edge corresponds
to taking inverse of a word. Given the data let $G(\G_n, d, j)$ 
be the quotient group of $F_d$, the free group on $d$ generators, by the relations
corresponding to the cycles in $\G_n$. The main result is that
there is a sequence of expander graphs $\G_n$ such that for a large enough 
$j$ the group $G(\G_n, d, j)$ is torsion-free, hyperbolic, and satisfies (a)
with probability $\to 1$ as $n\to\infty$.
Like any torsion-free hyperbolic group, $G(\G_n, d, j)$ has type {\it F}.

Izeki-Nayatani's original example is any uniform torsion-free lattice in $PSL_3(\mathbb Q_p)$,
which has type {\it F} because it acts freely and properly discontinuously on 
the associated Euclidean building.

\begin{rmk} The class of groups that satisfy (a) includes any finite group,
or more generally any locally finite infinite subgroup such as
$\mathbb Q/\mathbb Z$. There are much deeper examples in~\cite{ABJLMS} of 
finitely presented, infinite, non-torsion-free groups 
that fix a point for any action by a homeomorphisms
on a contractible manifold. None of these groups
satisfies (b) as they have nontrivial finite order elements.
\end{rmk}

\begin{rmk}
\label{rmk: refl gr trick}
By Davis's reflection group trick any group of type {\it F} embeds 
into the fundamental group of a closed aspherical manifold. Thus 
there is a closed aspherical manifold that is not homotopy equivalent
to a complete manifold of $K\le 0$.
\end{rmk}

\begin{quest}
Is there a closed aspherical manifold whose fundamental group satisfies \textup{(a)}?
\end{quest}

\section{Homeomorphism obstructions: exploiting $\mathbb R$ factors}
\label{sec: homeo obstr} 

Call a group $G$ {\it reductive\,} if for any epimorphism
$G\to H$ such that $H$ is a discrete, non-cocompact,
torsion-free isometry group of a Hadamard manifold stabilizing
a horoball, or a totally geodesic submanifold where $H$ acts cocompactly. 

\begin{ex}
\label{ex: elem}
(1) 
Any quotient of a reductive group is reductive.
\newline
(2)
Any finitely generated, virtually nilpotent group
is reductive (as follows from Sections~\ref{sec: non-parabolic}, 
\ref{sec: center with parabolic}, see~\cite{Bel-bus}).\newline
(3) Any irreducible, uniform lattice in the isometry group
of a symmetric space of $K\le 0$ and real rank $>1$ is reductive,
by the harmonic map superrigidity~\cite[Theorem 1.2]{Duc}, 
see also~\cite{Bel-bus}.
\end{ex}

A manifold is {\it covered by $\mathbb R^n$}
if its universal cover is diffeomorphic to $\mathbb R^n$.
Thus any complete $n$-manifold of $K\le 0$ is covered by $\mathbb R^n$.

\begin{rmk}
It is well-known that an open $K(G,1)$ manifold covered by 
$\mathbb R^n$ exists if and only if $G$ is a countable
group of finite cohomological dimension, see e.g.~\cite{Bel-bus}.
\end{rmk}

A manifold is {\it covered by $\mathbb R\times\mathbb R^{n-1}$}
if it is diffeomorphic to the product of $\mathbb R$ and
a manifold covered by $\mathbb R^{n-1}$. For instance,
if $G$ is a discrete torsion-free isometry group of a Hadamard
manifold that satisfies the condition (3) of Section~\ref{sec: rough}, 
then $M$ is covered by $\mathbb R\times\mathbb R^{n-1}$.
As we saw above (3) can be forced by purely 
algebraic assumptions on $\pi_1(M)$:

\begin{ex} 
If $M$ is a complete connected manifold of $K\le 0$ 
such that $\pi_1(M)$ is either clinging with finite abelianization,
or satisfies the assumptions of Theorem~\ref{thm: into-center}, then 
$M$ is covered by $\mathbb R\times\mathbb R^{n-1}$.
\end{ex}

A trivial method of producing manifolds that are covered by $\mathbb R^n$
but not covered by $\mathbb R\times\mathbb R^{n-1}$
is to consider any manifold of minimal dimension among all manifolds
in its homotopy type 
that are covered by a Euclidean space, which yields (see~\cite{Bel-bus}):

\begin{prop}
Any aspherical manifold is homotopy equivalent
to a manifold covered by $\mathbb R^n$ but
not covered by $\mathbb R\times\mathbb R^{n-1}$.
\end{prop}

This method is non-constructive
for it is not easy to decide whether a specific
open manifold has the minimal dimension in the above sense
(see~\cite{BKK02, BesFei02, Yoo04, Des06} for the
manifolds of such minimal dimensions). 

\begin{cor} 
\label{cor: inf gener}
If $G$ is reductive, clinging with finite abelianization,
or as in \textup{Theorem~\ref{thm: into-center}},
then any $K(G,1)$ manifold is homotopy equivalent to a manifold 
that admits no metric of $K\le 0$
and is covered by a Euclidean space.
\end{cor}

\begin{ex}
Corollary~\ref{cor: inf gener} applies if $G=\mathbb Q$, see Example~\ref{ex: rationals}.
\end{ex}

An essential tool in understanding  
manifolds covered by $\mathbb R\times\mathbb R^{n-1}$ 
is the recent result of
Guilbault~\cite{Gui-prod-R07}: if an open manifold $W$
of dimension $\ge 5$ is homotopy equivalent to a finite
complex, then $\mathbb R\times W$ is diffeomorphic
to the interior of a compact manifold.
Building  on this result, the author~\cite{Bel-bus}
proved

\begin{thm} 
\label{thm-intro: reg nbhd}
Let $W$ be an open $(n-1)$-manifold with $n\ge 5$ 
that is homotopy equivalent
to a finite complex of dimension $k\le n-3$.
Then $\mathbb R\times W$ is diffeomorphic to
the interior of a regular neighborhood of a $k\!$-dimensional
finite subcomplex. 
\end{thm}

With more work one gets~\cite{Bel-bus} the following applications:

\begin{thm}
\label{thm: R-factor conseq}
Let $L$ be a finite aspherical CW complex
such that $G=\pi_1(L)$ is reductive, clinging with finite abelianization,
or as in \textup{Theorem~\ref{thm: into-center}}. 
Suppose that $L$
is homotopy equivalent to a complete $n$-manifold $M$ of $K\le 0$ and $n\ge 5$,
and set $l=\dim(L)$. 
\newline
\textup{(1)}
If $l\le n-3$, then 
$M$ diffeomorphic to the interior of a regular neighborhood
 \phantom{(1)} of a $k\!$-dimensional finite subcomplex.\newline
\textup{(2)}
If $L$ is a closed manifold of dimension $<\frac{2n-2}{3}$,
then $M$ is diffeomorphic \phantom{(1)} to the total space of a vector bundle
over $L$.
\newline
\textup{(3)}
If $l<\frac{n}{2}$, then every complete
$n$-manifold of $K\le 0$ in the tangential \newline \phantom{(1)} homotopy type of $M$
is diffeomorphic to $M$.\newline
\textup{(4)}
If $l\le n-3$, then the tangential homotopy type of $M$ contains countably  \phantom{(1)} 
many open $n$-manifolds that admit no complete metric of $K\le 0$.
\end{thm}

\begin{quest}
Can one strengthen the conclusion 
``countably many'' in the part \textup{(4)} of \textup{Theorem~\ref{thm: R-factor conseq}}
to ``a continuum of''?
\end{quest}

A positive answer is given in~\cite{Bel-bus}
under a technical assumption 
which holds e.g. if $L$ is a closed manifold,
or if either $\mathbb Z^3$ or $\mathbb Z\ast \mathbb Z$
does not embed into $G$.

\begin{rmk}
Limiting Theorem~\ref{thm: R-factor conseq} to certain classes of manifolds of $K\le 0$
may result in enlarging the class of allowable fundamental groups, e.g. applying the
theorem to manifolds with visibility universal cover, we can allow $G$ to be the
product of any two nontrivial groups.
\end{rmk}

As an application of Theorem~\ref{thm-intro: reg nbhd}, we get 
the following characterization of $\mathbb R^n$:

\begin{cor} 
\label{cor-intro: Rn}
An open contractible $n\!$-manifold $W$ is homeomorphic
to $\mathbb R^n$ if and only if $W\times S^1$
admits a metric of $K\le 0$. 
\end{cor}

If $n=4$, then ``homeomorphic'' in 
Corollary~\ref{cor-intro: Rn} cannot be upgraded to
``diffeomorphic'': if $W$ is an exotic $\mathbb R^4$, then
$W\times S^1$ is diffeomorphic to 
$\mathbb R^4\times S^1$.

\section{Benefits of a lower curvature bound}

Complete manifolds of $\Ric\ge -(n-1)$ are central to the global
Riemannian geometry. For manifolds of $K\le 0$, a lower Ricci curvature bound at a point is equivalent
(by standard tensor algebra considerations)
to a lower sectional curvature bound at the same point; by rescaling one
can always make the bounds equal the curvature of the hyperbolic $n$-space.

In the seminal work~\cite{Gro-vol-bd-coh} Gromov uncovered a relation 
between the simplicial volume and volume growth, which for complete
manifolds with $\Ric\ge -(n-1)$ is governed by Bishop-Gromov
volume comparison. The following can be found 
in~\cite[p.13, 37]{Gro-vol-bd-coh}.

\begin{thm}
\label{thm: Gromov simp vol}
{\bf (Gromov)}\ \it
Let $W$ be an $n$-manifold such that
every component $C_i$ of $\d W$ is compact. 
If the interior of $W$ is homeomorphic to a complete manifold of $\Ric\ge -(n-1)$, 
then $\displaystyle{\sum_i ||C_i||}\le \displaystyle{\liminf_{r\to\infty}}\frac{\mathrm{Vol}\, B_p(r)}{r}$.
\end{thm}

Here $B_p(r)$ is the $r$-ball in $M$ centered at $p$, and $||C_i||$
is the simplicial volume of $C_i$.

\begin{ex}
If $M$ in Theorem~\ref{thm: Gromov simp vol} 
has finite volume, or more generally sublinear volume growth, then
each component of $\d W$ has zero simplicial volume.
\end{ex}

\begin{quest}
Does every complete manifold with $-1\le K\le 0$ and sublinear volume growth admit
a finite volume metric? 
\end{quest}

\begin{ex}
Any infinite cyclic cover a closed connected manifold $L$ of $K\le 0$, i.e. the cover
corresponding to the kernel of a surjective homomorphism $\pi_1(L)\to \mathbb Z$,
has linear growth.
\end{ex}

\begin{prob}
Study manifolds of $-1\le K\le 0$ with linear volume growth.
\end{prob}

Another consequence of a lower curvature bound 
is the famous Margulis lemma which appeared in~\cite{BGS} 
for manifolds of $-1\le K\le 0$ following unpublished ideas of Margulis, and
in~\cite{FukYam, KPT} for manifolds of $K\ge -1$. 
The following version of the Margulis lemma  for manifolds of $\Ric\ge -(n-1)$
is due to Kapovitch-Wilking~\cite{KapWil} with essential ingredients provided 
by prior works of Cheeger-Colding.

\begin{thm}
\label{thm: wilk-kap margulis}
\bf (Kapovitch-Wilking)\ \it
For each $n$ there are constants $m$ and $\e\in (0,1)$
such that if $p$ is a point of a complete $n$-manifold with $\Ric\ge -(n-1)$ 
on $B_p(1)$, then the image of $\pi_1(B_p(\e))\to\pi_1(B_p(1))$ has a nilpotent subgroup
generated by $n$ elements and of index $\le m$.
\end{thm}

In fact, the nilpotent subgroup in Theorem~\ref{thm: wilk-kap margulis} 
has a generated set $\{s_1, \dots, s_n\}$ 
such that $s_1$ is central and the commutator $[s_i, s_j]$ is contained in the
subgroup generated by $s_1, \dots, s_{i-1}$ for each $1<i<j$.
Another universal bound on the number of generators of any given $r$-ball is given by

\begin{thm} \label{thm: wilk-kap fg}
\bf (Kapovitch-Wilking)\ \it
For each $n$, $r$ there is a constant $k$ such that
if $p$ is a point in a complete Riemannian manifold $M$
such that $\pi_1(B_p(r))\to\pi_1(M,p)$ is onto and
$\Ric\ge -(n-1)$ on $B_p(4r)$, then $\pi_1(M,p)$
is generated by $\le k$ elements.
\end{thm}

An important feature of the two preceding results is that 
no curvature control is required outside a compact subset.

\section{Injectivity radius going to zero at infinity}

We say that a subset $S$ of a Riemannian manifold 
{\it has ${Inj}\,{Rad}\to 0$\,} 
if and only if for every $\e>0$
the set of points of $S$ with injectivity radius $\ge \e$ is compact;
otherwise, we say $S$ {\it has ${Inj}\,{Rad}\not\to 0$}.

\begin{rmk}
By volume comparison
any finite volume complete manifold of $K\le 0$ has
$\mathrm{Inj}\,\mathrm{Rad}\to 0$~\cite[8.4]{BGS}.
\end{rmk}

\begin{prop} 
\label{prop: surface with injrad go to zero}
Any finite volume complete real hyperbolic manifold
admits a complete metric with $\mathrm{Inj}\,\mathrm{Rad}\to 0$, 
bounded negative curvature, infinite volume, and sublinear volume growth. 
\end{prop}
\begin{proof} We just do the two dimensional case; the general case is similar. 
Any end of a finite volume complete hyperbolic surface 
surface has an annular neighborhood
with the metric $dt^2+e^{-2t}d\phi^2$, $t>0$. Modify
it to the metric $dt^2+f^2(t)d\phi^2$ where $f$ is a convex decreasing
function such that $f(t)=e^{-t}$ for small $t$,
and $f(t)=t^{-\a}$, $\a\in (0,1)$ for large $t$. Let $\Sigma_f$ be 
the resulting complete Riemannian $2$-manifold,
and let $\Sigma_f^r$ denote ``$\Sigma_f$
with the portion with $t>r$ chopped off''. Now $\Sigma_f$ has \newline
$\bullet$
$\mathrm{Inj}\,\mathrm{Rad}\to 0$ because $f$ monotonically decreases to zero,\newline
$\bullet$
infinite volume since $\frac{1}{2\pi}\vol(\Sigma_f^r)$
grows (sublinearly) as $\int_0^r f(s)ds=\frac{r^{1-\a}}{1-\a}$,\newline
$\bullet$
bounded negative sectional curvature because on the annular neighborhood \newline\phantom{$\bullet$}
$K=-\frac{f^{\prime\prime}}{f}<0$ which equals
$-\frac{\a(\a+1)}{t^2}$ for $t>r$.
\end{proof}

\begin{quest}
Is there a complete manifold of $K\le 0$ and $\mathrm{Inj}\,\mathrm{Rad}\to 0$
that admits no complete finite volume of $K\le 0$? 
What is the answer in the presence of a lower curvature bound.
\end{quest}

Gromov~\cite{Gro-jdg78} pioneered the study of ends of negatively curved manifolds
via the critical point theory for distance functions, which was extended by 
Schroeder~\cite[Appendix 2]{BGS} as follows:

\begin{thm}{\bf (Schroeder)}
\label{thm: schroeder}
If $M$ is a complete manifold of $\mathrm{Inj}\,\mathrm{Rad}\to 0$ and $-1\le K\le 0$, then
either $M$ is the interior of a compact manifold,
or $M$ contains a sequence of totally geodesic, immersed, flat
tori with diameters approaching zero. 
\end{thm}

None of the assumptions in the above theorem
can be dropped due to examples of Gromov~\cite[Chapter 11]{BGS}
and Nguyen Phan~\cite{Pha-neg} in which $\pi_1(M)$ is infinitely generated, 
see Section~\ref{sec: inf hom}. 

\begin{prob}
Find a geometrically meaningful compactification of 
complete manifolds of $\mathrm{Inj}\,\mathrm{Rad}\to 0$ and $-1\le K\le 0$.
\end{prob}

In the locally symmetric case this was accomplished 
in~\cite{Leu-invent, Sap-compactif, Leu-diff-geom-appl}, and
pinched negatively curved manifolds are naturally compactified by
horospheres.

A weak substitute for a geometrically meaningful compactification is given
by the following general theorem~\cite{CheGro-chop}:

\begin{thm}{\bf (Cheeger-Gromov)}\ \it
\label{thm: cheeger-gromov}
For each $n$ there is a constant $c$ such that any
complete finite volume $n$-manifold $M$ of $|K|\le 1$ 
admits an exhaustion by compact smooth codimension zero submanifolds
$M_i$ with boundary  such that $M_i\subset \mathrm{Int}\,(M_{i+1})$,
the norm of the second fundamental form of $\d M_i$ is $\le c$,
and $\vol\,(\d M_i)\to 0$ as $i\to \infty$.
\end{thm}

The proof constructs a controlled exhaustion function on $M$.
For a related work based on different technical tools
see~\cite[Theorem I.4.2]{SchYau-diff-geom-book}
and~\cite{Daf-exh, WanLin-exh}.

If $M$ in Theorem~\ref{thm: cheeger-gromov} 
is the interior of a manifold $W$ with compact boundary, then
considering the components of $\d M_i$ that lie
in a collar neighborhood of $\d W$, we conclude that 

\begin{cor}
\label{cor: cheeger-gromov}
If $W$ is a manifold with compact connected boundary whose interior
admits a complete finite volume metric $g$ of $|K|\le 1$, then
$\mathbb R\times\d W$
contains the sequence of compact separating hypersurfaces $H_i$ which are
homologous to $\{0\}\times\d W$ and satisfy $\mathrm{MinVol}(H_i)\to 0$ as $i\to\infty$.
Moreover,\newline\phantom{(1)}
\textup{(1)} $\d W$ has even Euler characteristic and zero simplicial volume;
\newline\phantom{(1)}
\textup{(2)}
If $\d W$ is orientable, then its Pontryagin numbers vanish.
\newline\phantom{(1)}
\textup{(3)} If $g$ also has $K\le 0$, then
the $\ell^2$-Betti numbers of 
$\d W$\newline\phantom{(3)\,}\phantom{(3)} vanish, and hence $\d W$ has zero
Euler characteristic.
\end{cor}
\begin{proof}
Projecting onto the $\d W$ factor yields a degree one map
$\d M_i\to \d W$ (if $\d W$ is non-orientable, then so is $H_i$,
and we get a degree one map of their orientation covers).
Thus $\d W$ has zero simplicial volume, which of course
we already knew by Theorem~\ref{thm: Gromov simp vol}. 

By Chern-Weil theory the Pontryagin numbers $p_I(L)$ 
of a closed manifold
$L$ satisfy $|p_I(L)|\le c_l\, \mathrm{MinVol}(L)$ where $c_l$ is a constant
depending only on $l=\dim(L)$~\cite{Gro-vol-bd-coh}. Thus $p_I(H_i)\to 0$ as $i\to\infty$.
Since Pontryagin numbers
are oriented cobordism invariant, we conclude that if $\d W$
in Corollary~\ref{cor: cheeger-gromov} is orientable, then
its Pontryagin numbers vanish.

The boundary of a compact manifold has even Euler 
characteristic~\cite[Corollary VIII.8.8]{Dol-book}, so applying
this to the cobordism between $\d W$ and $H_i$ we see that
$\chi(\d W)+\chi(H_i)$ is even, and again Chern-Weil theory
implies $\chi(H_i)\to 0$ as $i\to\infty$, and the claim follows.

Finally, vanishing of the $\ell^2$-Betti numbers follows 
from~\cite[Theorem 1.2]{CheGro-vndim}, and their alternating sum
equals the Euler characteristic.
\end{proof}

\begin{quest} Do any of the conclusions \textup{Corollary~\ref{cor: cheeger-gromov}}
hold for complete manifolds with $-1\le K\le 0$  
and $\mathrm{Inj}\,\mathrm{Rad}\to 0$. 
\end{quest}

\begin{ex} If $\Sigma_f$ is as in
Proposition~\ref{prop: surface with injrad go to zero}, then
the Riemannian product
$\Sigma_f\times\Sigma_f$ has $\mathrm{Inj}\,\mathrm{Rad}\to 0$ but superlinear
volume growth if $0<\a\le\frac{1}{2}$ because
for large $r$ the subset $\Sigma_f^r\times\Sigma_f^r$ 
is sandwiched between concentric balls of radii $r$ and $3r$, 
and its volume grows superlinearly as $\a\le\frac{1}{2}$.
Thus proving that the boundary has zero simplicial volume
one requires new ideas beyond Theorem~\ref{thm: Gromov simp vol}.
\end{ex}

\section{Negatively curved manifolds with uniform volume bound}

For a connected complete Riemannian manifold $M$ we denote by $\widetilde M$ its universal
cover with the pullback metric. 

Fukaya~\cite{Fuk-finiteness-neg} proved the following result,
whose analog for closed manifolds of dimension $\neq 3$
is due to Gromov~\cite{Gro-jdg78}:

\begin{thm}
\label{thm: fukaya finiteness}
\bf (Fukaya)\ \it
Given $V$ and $n\neq 3, 4$, 
only finitely many of diffeomorphism
classes contain open complete $n$-manifolds
$M$ such that\vspace{1pt}\newline
\textup{(1)} $K<0$ or $\widetilde M$ is visibility,\vspace{1pt}\newline
\textup{(2)} $K\ge -1$, \vspace{1pt}\newline
\textup{(3)} $\vol\,(M)< V$.
\end{thm}

In dimension four Fukaya proved that the class of manifolds satisfying (1)-(3)
contains only finitely many homotopy types (the missing ingredient 
is the weak h-cobordism theorem, which is unknown for 
h-cobordisms between closed $3$-manifolds).

The theorem fails in dimension three as 
there are infinitely many (both open and closed) hyperbolic $3$-manifolds 
with uniformly bounded volume~\cite{Thu-3d-notes}.
Taking products with flat tori demonstrates that
(1) cannot be replaced with $K\le 0$, even though 
the optimal curvature condition is unclear.

\begin{quest}
Is \textup{Theorem~\ref{thm: fukaya finiteness}} true with
\textup{(1)} replaced by ``$\widetilde M$ has rank one'',
or ``$\widetilde M$ contains no flat half planes''?
\end{quest}

\begin{rmk}
The proof in~\cite{Gro-jdg78, Fuk-finiteness-neg}
established an upper diameter bound in terms of volume, and
then applies Cheeger's finiteness theorem 
(if $M$ is open the diameter bound is for a compact domain
$D$ such that $M\setminus D$ is the interior of 
an h-cobordism). The strategy fails
if one merely assumes that $\widetilde M$ has rank one
by the following example: 
Chop off a cusp of a finite volume complete real hyperbolic manifold,
and modify the metric to have totally geodesic flat boundary and $K\le 0$.
Then double along the boundary, which gives a finite volume complete
rank one manifold of $K\le 0$ and volume bounded roughly by $2\vol(M)$, 
but its diameter can be chosen arbitrary large by chopping 
deeper into the cusp.
\end{rmk}

\begin{quest}
How does the number of diffeomorphism types of manifolds
$M$ in \textup{Theorem~\ref{thm: fukaya finiteness}} grows with $n$ and $V$?
\end{quest}

In the locally symmetric case the above question
was extensively studied, see \cite{Gel-vol-growth} and references therein.

\section{Non-aspherical ends of nonpositively curved manifolds}

If a (not necessarily connected) manifold $B$ is diffeomorphic to the 
boundary of a connected, smooth (not necessarily compact) manifold $W$, then we say 
that $B$ {\it bounds} $W$. 

Any aspherical manifold $B$ bounds
a noncompact aspherical manifold, namely
$B\times [0,1)$, and in fact, the universal cover
of $B\times (0,1)$ is a Euclidean space.
Note that $B\times (0,1)$ admits
a complete metric of $K\le 0$ if 
$B$ is an infranilmanifold~\cite{BK-GAFA}, or if
$B$ itself admits a complete metric of $K\le 0$. 
On the other hand, if $\pi_1(B)$ contains a subgroup
with strong fixed point properties 
as in Section~\ref{sec: homotopy obstr}, then 
$B\times (0,1)$ admits no complete metric of $K\le 0$. 
Our ignorance is illustrated by the following

\begin{quest}
Does every closed
aspherical manifold bounds a manifold 
whose interior admits a finite volume 
complete metric of $K\le 0$?
\end{quest}

In this section we discuss similar matters when $B$ is closed
and not aspherical. We focus on easy-to-state results and refer
to~\cite{BelPha-enp} for a complete account.

Boundaries of compact manifolds with a complete metric
of $K\le 0$ on the interior could be quite diverse:

\begin{ex}\ \newline
(1) The total space of any vector bundle a closed
manifold of $K\le 0$ admits a complete metric of $K\le 0$~\cite{And-vb}.\vspace{1pt}\newline
(2)
Complete finite volume locally symmetric manifold of $K\le 0$ and 
$\mathbb Q$-rank $\ge 3$ are interiors 
of compact manifolds with non-aspherical boundary.\vspace{2pt}\newline
(3)
A complete manifold $M$ of $K\le 0$ is {\it convex-cocompact\,} if it 
deformation retracts onto a compact locally convex
subset; such $M$ is the interior of a compact manifold 
whose boundary is often non-aspherical. 
\end{ex}

There seem to be no simple description of closed manifolds
that bound aspherical ones, and some obstructions are summarized below. 
In order for $B$ to bound an aspherical manifold, 
a certain covering space of $B$ must bound a 
contractible manifold. In formalizing how this restricts 
the topology of $B$, the following definition is helpful:
given a class of groups $\mathcal Q$, a group 
is {\it anti\,}--$\mathcal Q$ if it admits no nontrivial
homomorphism into a group in $\mathcal Q$. 
Clearly, the class of anti--$\mathcal Q$ groups
is closed under extensions, quotients,
and any group generated by a family of
anti--$\mathcal Q$ subgroups is anti--$\mathcal Q$.

\begin{ex} Let $\mathcal A_n$ denote the class of fundamental groups of aspherical $n$-manifolds.
See~\cite{BelPha-enp} for examples of anti--$\mathcal A_n$ groups in such as:
\begin{enumerate}
\item Any group generated by a set of finite order elements.
\item
Any irreducible lattice in the isometry group of
a symmetric space of rank $\ge 2$ and dimension $>n$~\cite{BesFei02}. 
\end{enumerate}
\end{ex}

The following summarizes some obstructions that prevent
a manifold from bounding an aspherical one. 

\begin{thm}\label{thm: intro-asp}
If $B$ bounds an aspherical, non-contractible $n$-manifold, 
and $\pi_1(B)$ is anti--$\mathcal A_n$, then $B$ is noncompact,
parallelizable, its $\mathbb Z$-valued intersection
form of vanishes, and  
its $\mathbb Q/\mathbb Z$-valued torsion linking form 
vanishes.  
\end{thm}

\begin{ex}
The following manifolds do not bound aspherical ones:
\begin{enumerate}
\item The connected sum of lens spaces, because it is a closed
manifold whose fundamental group is anti--$\mathcal A_n$.
\item The product of any manifold with 
$CP^k$ with $k\ge 2$.
\item The connected sum of any manifold and 
the product of two closed manifolds whose fundamental
groups are anti--$\mathcal A_n$.
\item
The product of a punctured 
$3$-dimensional lens space and a closed manifold
whose fundamental group is anti--$\mathcal A_n$.
\item 
Any manifold that contains the manifold in (2)--(4)
as an open subset.
\end{enumerate}
\end{ex}

Let $\mathcal{NP}_n$ denote the class of the fundamental groups of complete
$n$-manifolds of $K\le 0$; of course $\mathcal{NP}_n\subseteq\mathcal{A}_n$.
Examples of anti-$\mathcal{NP}_n$ groups of type {\it F} 
discussed in Section~\ref{sec: homotopy obstr}
immediately imply (see ~\cite{BelPha-enp}):

\begin{thm}
\label{thm: anti-NP-Gromov}
There is a closed non-aspherical manifold that\vspace{1pt}\newline
\textup{(i)}
bounds a manifold whose interior is covered by a Euclidean space;\vspace{1pt}\newline
\textup{(ii)}
bounds no manifold whose interior 
has a complete metric of $K\le 0$.
\end{thm}

Other obstructions come from results of Section~\ref{sec: homeo obstr}.
Given groups $I$, $J$ and a class of groups $\mathcal Q$ we say that
{\it $I$ reduces to $J$ relative to $\mathcal Q$}
if every homomorphism $I\to Q$ with $Q\in\mathcal Q$   
factors as a composite of an epimorphism $I\to J$ 
and a homomorphism $J\to Q$. Here we are mainly interested in
groups that reduce relative to $\mathcal{NP}_n$ to the groups from
the parts (2)-(3) of Example~\ref{ex: elem}, which have finite
virtual cohomological dimension.

\begin{thm}
\label{thm-intro: codim >2}
Let $n\ge 6$, let $G$ be a group from \textup{Example~\ref{ex: elem}(2)-(3)}
of virtual cohomological dimension $\le n-3$, and let $B$ be a closed 
$(n-1)$-manifold such that $\pi_1(B)$ 
reduces to $G$ relative to $\mathcal{NP}_n$.
If $B$ bounds a manifold $N$ such that $\mathrm{Int}(N)$ admits
a complete metric of $K\le 0$, then there is a closed manifold $L$ of dimension $\le n-3$
such that\vspace{-3pt}
\begin{itemize}
\item[(1)]
$L$ is either an infranilmanifold, or an irreducible, locally symmetric manifold
of $K\le 0$ and real rank $\ge 2$.
\item[(2)]
$N$ is the regular neighborhood of a PL-embedded copy of $L$.
\item[(3)]
If $N$ is not diffeomorphic to the 
product of a compact manifold and a closed interval, then
$L$ admits a metric of $K\le 0$ and $N$ is the total space of a linear 
disk bundle over $L$.
\end{itemize}
\end{thm}

In~\cite{BelPha-enp} we give examples of manifolds $B$ that cannot bound
a manifold $N$ as in the above theorem such as 

\begin{ex} Let $B$ be the total space of a linear
$S^k$ bundle over a closed non-flat infranilmanifold such that $k\ge 3$
and the rational Euler class of the bundle is nonzero.
Then $B$ does not bound a manifold whose interior admits a complete metric of $K\le 0$.
\end{ex}

A boundary component of a manifold
is {\it incompressible\,} if its inclusion induces injections on
all homotopy groups. Reductive groups were defined in Section~\ref{sec: homeo obstr}.

\begin{thm} 
\label{thm: tors-free quot}
Let $B$ be a closed $(n-1)$-manifold such that
$\pi_1(B)$ is reductive and any nontrivial quotient of $\pi_1(B)$ in the class
$\mathcal{NP}_n$ has cohomological dimension $n-1$. 
If $B$ bounds a manifold $N$ such that $\mathrm{Int}(N)$ admits
a complete metric of $K\le 0$, then $B$ is incompressible in $N$. 
\end{thm}

\begin{ex} 
\label{ex: tf general}  
Theorem~\ref{thm: tors-free quot} applies to whenever $\pi_1(B)$
is isomorphic to $\pi_1(L)$ where 
$L$ is an $(n-1)$-manifold of from Theorem~\ref{thm-intro: codim >2}(1) 
and $\pi_1(L)$ has no proper torsion-free quotients.
Examples of such $L$ include any higher rank,
irreducible, locally symmetric manifolds of $K\le 0$
(thanks to the Margulis Normal Subgroup Theorem), as well as certain 
infranilmanifolds, see~\cite{BelPha-enp}.
\end{ex}

Given a compact boundary component $B$ of a manifold $N$,
an {\it end $E$ of $\mathrm{Int}(N)$ that corresponds to $B$} 
is the intersection of $\Int(N)$
with a closed collar neighborhood of $B$; note that 
$E$ is diffeomorphic to $[1,\infty)\times B$.

\begin{thm}
\label{thm: tg}
Let $B$ be a closed connected manifold that
bounds a manifold $N$, and let $E$ be an end of $\mathrm{Int}(N)$
corresponding to $B$. 
If $\pi_1(B)$ is reductive, and 
$\mathrm{Int}(N)$ admits a complete metric of $K\le 0$ and
$\mathrm{Inj}\,\mathrm{Rad}\to 0$ on $E$, then
$B$ is incompressible in $N$.
\end{thm}

\begin{ex}
\label{ex: sphere bundle}
If $B$ is the total space of a linear $S^k$ bundle with $k\ge 2$ over
a manifold $L$ as in~\ref{thm-intro: codim >2}(1), then 
$B$  does not bound a manifold 
whose interior admits a complete metric
of $K\le 0$ and
$\mathrm{Inj}\,\mathrm{Rad}\to 0$. 
\end{ex}

\section{Riemannian hyperbolization (after Ontaneda)} 
\label{sec: ontaneda}

A recent work of Ontaneda~\cite{Ont-smth-hyp} 
allows to dramatically expand the list of known finite volume
complete manifolds of $-1\le K<0$ of dimensions $>3$. Unlike earlier examples,
Ontaneda's method assembles a manifold of $K\le -1$ in a lego-like
fashion from identical blocks according to a combinatorial pattern
specified by a cube complex structure on any given manifold. 
Each block is a compact real hyperbolic manifold
with corners, where every boundary face is totally
geodesic and the faces's combinatorial pattern is that of a cube. 
This process results in a singular metric, which Ontaneda is able
to smooth into a complete Riemannian metric of $K\le -1$, provided
the block's faces have a sufficiently large normal
injectivity radii; such blocks exist.

The idea of building locally CAT($0$) manifold out of identical blocks
is due to Gromov~\cite{Gro-hypgr} who came up with several hyperbolization procedures
turning a simplicial complex into a locally CAT($0$) cubical complex.
Gromov's ideas were developed and made precise in~\cite{DavJan-jdg-hyperb, DJW}, cf.~\cite{Dav-book},
and furthermore Charney-Davis~\cite{ChaDav} developed the 
{\it strict hyperbolization\,} that
turns a cubical complex into a piecewise polyhedral complex whose
faces are the blocks of the previous paragraph. 
A key feature of the procedures is that 
the link of every vertex of the hyperbolized polyhedral complex is a subdivision of
the corresponding link in the original complex, which has
two consequences:
\begin{itemize}
\item
The hyperbolization(s) turns a manifold triangulation into 
a locally CAT($0$) manifold with a  cubical complex structure;
\item
The strict hyperbolization turns a locally CAT($0$) cubical complex (or manifold)
into a locally CAT($-1$) polyhedral complex (manifold, respectively).
\end{itemize}
Trivial exceptions aside, the resulting piecewise Euclidean 
(or piecewise hyperbolic) metric is non-Riemannian.
Smoothing the piecewise hyperbolic metric into a Riemannian metric 
of $K\le -1$ in~\cite{Ont-smth-hyp} 
is a technological tour de force. 

\begin{rmk}\ \begin{enumerate}
\item
The cube complex fed into Ontaneda's construction
need not be locally CAT($0$), so the resulting manifolds of $K\le -1$ in~\cite{Ont-smth-hyp}
are a priori 
not homeomorphic to locally CAT($-1$) manifolds of~\cite{ChaDav}.
\item
Charney-Davis~\cite{ChaDav} describe a canonical smoothing of their manifolds
(but not of the metrics), yet this smoothing is not necessarily equal to
the smooth structure in~\cite{Ont-smth-hyp} even when there is a face
preserving homeomorphism between the two manifolds.
\item 
Taking the normal injectivity radius
of the block's faces large enough, one can make the sectional
curvature of the resulting manifold arbitrary close to $-1$.
\item
If one starts with a closed manifold,
then Ontaneda's procedure yields a closed manifold of $K\le -1$, 
but if the initial manifold is open, its triangulation contains infinitely many simplices,
so the resulting Riemannian manifold of $K\le -1$ has infinite volume.\end{enumerate}
\end{rmk} 

In order to produce finite volume examples,
Ontaneda relativizes the construction as follows:
Start with a closed manifold $B$ that bounds a compact manifold $W$, 
cone off the boundary, apply the strict hyperbolization, and
remove the cone point. The result is a compact manifold
with boundary, and Ontaneda gives it a certain smooth
structure in which the boundary becomes diffeomorphic to $B$.
Topological properties of the resulting smooth manifold $R_{W,B}$ 
mirror those of $W$, see a summary in~\cite{Bel-hyper-polyh}, and in particular
by varying $W$, one finds $R_{W,B}$'s such that
\begin{itemize}
\item  $R_{W,B}$ has a  nontrivial rational Pontryagin class if $\dim(W)\ge 4$,
\item $H^*(R_{W,B})$ contains a subring isomorphic to the cohomology ring of a
given finite CW complex of dimension $<\dim(W)/2$, see~\cite{Ont-smth-hyp},
\item  $\pi_1(R_{W,B})$ surjects on a given finitely presented group if $\dim(W)\ge 4$.
\end{itemize}
Moreover $R_{W,B}$ enjoys the following properties:
\begin{itemize}
\item $R_{W,B}$ is orientable if so is $W$,
\item the inclusion of each component of $B$ into $R_{W,B}$ is $\pi_1$-injective~\cite{DJW},
\item $R_{W,B}$ is aspherical if and only if each component of $B$ is aspherical~\cite{DJW},
\item $\pi_1(R_{W,B})$ is hyperbolic relative to the images of the
fundamental groups of components of $B$~\cite{Bel-hyper-polyh}.
\end{itemize}
The piecewise hyperbolic metric on $R_{W,B}$ is singular and incomplete near
the removed cone point, but again when the normal injectivity radius
of the block's faces large enough, Ontaneda is able to smooth the metric 
away from a punctured neighborhood of the cone point, while on that
neighborhood the metric has to be constructed by an ad hoc
method depending on $B$.
The following result is implicit in~\cite{Ont-smth-hyp}. 

\begin{thm} 
\label{thm-intro: ont}
{\bf (Ontaneda)}\ Let $B$ be the boundary of a compact $n$-manifold and suppose
that 
\newline \textup{(i)}
if $n\ge 6$, then any h-cobordism from $B$ to another manifold is a product;\newline
\textup{(ii)}  $\mathbb R\times B$  admits 
a complete metric $g$ of sectional curvature within 
$[-1, 0)$ such  \phantom{\textup{(ii)}}
that $(-\infty, 0\,]\times B$ has finite $g$-volume, and
$g=dr^2+e^{2r} g_{_B}$ on $[c,\infty)\times B$  \phantom{\textup{(ii)}} for some $c>0$ and
a metric $g_B$ on $B$. \newline
Then $B$ bounds a compact smooth manifold whose interior
admits a complete metric of finite volume and
sectional curvature in $[-1,0)$. 
\end{thm}

Condition (ii) implies that each component of
$B$ is aspherical, and hence has torsion-free fundamental group. 
The Whitehead Torsion Conjecture, which is true for many groups of geometric 
origin~\cite{BarLue-borelconj, BFL}, 
predicts that all torsion-free groups have 
zero Whitehead torsion. If the conjecture is true for
the fundamental group of each component of $B$, then (i) holds.
Condition (ii) has been checked in a number of cases 
in~\cite{BK-GAFA, Pha-sol, Bel-ewp} implying:

\begin{thm}
\label{thm: sol class B}
A manifold $B$ bounds a compact manifold whose interior
admits a complete metric of finite volume and
sectional curvature in $[-1,0)$ if\newline
\textup{(1)}
$B$ is a closed $3$-dimensional $Sol$ manifolds~\cite{Pha-sol}, or
\newline
\textup{(2)} $B$ bounds a compact manifold, and belongs to the class
$\mathcal B$~\cite{Bel-ewp}.
\end{thm}

Here $\mathcal B$ is the smallest class of closed manifolds 
of positive dimension such that
\begin{itemize}
\item
$\mathcal B$ contains each infranilmanifold~\cite{BK-GAFA},
each closed manifold of $K\le 0$ with a local 
Euclidean de Rham factor of positive dimension, and 
every circle bundle of type  \textup{(K)};
\item $\mathcal B$ is closed under products, 
disjoint unions, and products with any compact 
manifold of $K\le 0$.
\end{itemize}

An orientable circle bundle {\it has type \textup{(K)}} if 
its base is a closed complex hyperbolic $n$-manifold whose
holonomy representation lifts from $PU(n,1)$ to $U(n, 1)$, and if the Euler class of the bundle
equals $-m\frac{\omega}{4\pi}$ for some nonzero integer $m$, where
$\omega$ is the K\"ahler form of the base. For example, every nontrivial orientable circle
bundle over a genus two orientable closed surface
has type (K), see~\cite{Bel-ewp}. 

\begin{quest}
Consider the class of closed manifolds $B$ to which \textup{Theorem~\ref{thm-intro: ont}} applies.
Does the class contain every circle bundle over a closed manifold of $K<0$?
Every closed infrasolvmanifold? Is the class closed under products?
\end{quest}

\section{Topology of known manifolds of bounded negative curvature}

In dimensions $>3$ most (if not all) known examples of complete manifolds
of $K\le 0$ come from combining 
\begin{itemize}
\item
locally symmetric metrics of $K\le 0$ arising from
arithmetic or reflection groups,
\item
iterated and multiple warped products, building
on the seminar work of Bishop-O'Neill on 
singly warped products of $K\le 0$~\cite{BisON-warp}.
\end{itemize} 
Ontaneda's work~\cite{Ont-smth-hyp} illustrates
how the two methods combine: the block is an arithmetic real hyperbolic manifold
with corners, and smoothing the metric involves sophisticated 
warped product considerations.

Prior examples of finite volume complete manifold of $K\le 0$
that are not locally symmetric were typically produced by
starting with a locally symmetric complete manifold of $K\le 0$ and performing
one of the following operations, which can be combined or iterated:
\begin{itemize}
\item
doubles (Heintze, see~\cite{Sch-warp}) and twisted doubles~\cite{Ont-double03},
\item
branched 
covers~\cite{MosSiu, GroThu, ForSch90, Zhe96, Ard-thesis00, Der05, FO-branch06},
\item
cusp closing~\cite{Sch-cusp-close, BuyKob-cusp-close, HumSch, And-cusps-einst},
\item
cut and paste a tubular neighborhood of a totally geodesic
submanifold of codimension $1$, or dimensions $0$ or $1$, see~\cite{FJO-survey}
for a survey,
\item
remove a family of totally geodesic submanifolds of codimension 
two, which results in an incomplete metric that in some cases can be modified
to a complete metric of $K\le 0$~\cite{Fuj-warp, AbrSch, Buy, Bel-ch-warp, Bel-rh-warp}.
\end{itemize}

With few general results available, it makes sense to study topology of 
known examples in more detail. For the rest of the section let
$M$ be an open connected complete finite volume manifold of $-1\le K<0$.
The prior discussions gives
\begin{itemize}
\item (Section~\ref{sec: acyl hyp}) 
$\pi_1(M)$ is acylindrically hyperbolic because $M$ has finite volume and $K<0$,
\item (Theorem~\ref{thm: schroeder}) $M$ is the interior of a compact manifold $N$, which is uniquely determined
up to attaching an h-cobordism to the boundary.
\item (Corollary~\ref{cor: cheeger-gromov}) $\d N$ has zero simplicial volume and Euler characteristic, and 
if $\d N$ is orientable, $\d N$ has zero Pontryagin numbers. 
\end{itemize}

We would get a lot more information
if $\pi_1(M)$ were hyperbolic relative to a collection of 
easy-to-understand peripheral subgroups (mainly because 
relatively hyperbolic groups inherit many properties from their
peripheral subgroups, see below). 
Here is a prototypical example:

\begin{ex}
If $M$ is negatively pinched
(i.e  $K$ is bounded between two negative constants), then $\pi_1(N)$
is hyperbolic relatively the fundamental groups of the components
of $\d N$~\cite{Far-rel-hyp, Bow-rel-hyp} 
which are virtually nilpotent~\cite{BGS}.
\end{ex}

One might expect that
$\pi_1(M)\cong\pi_1(N)$ is always hyperbolic relative to the fundamental
groups of components of $\d N$. This idea runs into difficulties because
$\d N$ need not be $\pi_1$-incompressible~\cite{Buy}, but when
it works it does so to great effect, and become a major source of information 
about $\pi_1(M)$.

We refer to~\cite{Osi-rel-hyp} for background on relatively
hyperbolic groups; as a part of the definition we require that
relatively hyperbolic groups are
finitely generated and not virtually cyclic, and their
peripheral subgroups are infinite and proper.

\begin{thm}
\label{thm: co-Hopf} \bf (Belegradek)\ \it
If $W$ is a compact aspherical manifold of dimension $\ge 3$ such that 
the components of $\d W$ are aspherical and $\pi_1$-injectively embedded,
and $\pi_1(W)$ is hyperbolic relatively to the fundamental groups 
of the components of $\d W$, then \vspace{-5pt}
\begin{enumerate}[(a)]\rm
\item \it
if $\pi_1(W)$ splits nontrivially as amalgamated product or
HNN extension over a subgroup $K$, then $K$ is acylindrically hyperbolic.
\rm \item \it
$\pi_1(W)$ is co-Hopf, i.e. its injective endomorphisms are surjective.
\rm\item \it
$\mathrm{Out}(\pi_1(W))$ is finite if for every component $B$ of $\d W$
the group $\pi_1(B)$ is either not relatively hyperbolic
or has a relatively hyperbolic structure whose peripheral subgroups
are not relatively hyperbolic.
\end{enumerate}
\end{thm}
\begin{proof}[Sketch of the proof] 
Mayer-Vietoris sequence in the group cohomology applied to the fundamental
group of the double of $W$ along $\d W$ can be used to show that 
$\pi_1(W)$ can only split over a non-elementary subgroup proving (a),
which in turn implies (b)--(c) by results of Dru{\c{t}}u-Sapir~\cite{DruSap-split}.
An alternative proof of (b) is immediate from the fact that 
$(W,\d W)$ has positive simplicial norm~\cite{MinYam}.
Details can be found in~\cite{Bel-hyper-polyh}.
\end{proof}

\begin{ex}
The fundamental group of a closed aspherical manifold with zero simplicial volume
is not relatively hyperbolic (as noted in~\cite{BelHru} this follows from~\cite{MinYam}).
Thus if $W$ is as in Theorem~\ref{thm: co-Hopf} and $\mathrm{Int}(W)$ 
admits a complete metric finite volume of $\Ric\ge -(n-1)$,
then $\d W$ has zero simplicial volume by Theorem~\ref{thm: Gromov simp vol},
so $\mathrm{Out}(\pi_1(W))$ is finite.
\end{ex}

Proving relative hyperbolicity was the main goal of the 
author's work in~\cite{Bel-hyper-polyh, Bel-ch-warp, Bel-rh-warp, BelHru}
where it was accomplished in the following cases:
\begin{ex}
\label{ex: rel hyp fin vol}\ \newline
(1)
If $R_{W,B}$ is from Section~\ref{sec: ontaneda},
then $\pi_1(R_{W,B})$
is hyperbolic relative to the fundamental groups
of components of $B$~\cite{Bel-hyper-polyh}. If $B$
is as in Theorem~\ref{thm: sol class B}, then $R_{W,B}$
admits a complete finite volume metric of $-1\le K<0$. 
\newline
(2) If  $V$ is a closed manifold of $K<0$, 
and $S$ is an embedded, codimension two,
compact, totally geodesic submanifold, then
$\pi_1(V\setminus S)$ is hyperbolic relative
to the fundamental groups of boundary components of a 
tubular neighborhood of $S$ in $V$~\cite{BelHru}.
A finite volume complete metric 
of $-1\le K<0$ on $V\setminus S$ was constructed 
in~\cite{Bel-ch-warp, Bel-rh-warp} 
when either $V$ is real hyperbolic, or $V$ and $S$ are 
complex hyperbolic.\newline
(3)
Let $V$ be a closed manifold of $K<0$, 
and $S$ be an immersed, codimension two,
compact, totally geodesic submanifold 
whose preimage to the universal cover $\widetilde V$
of $V$ is {\it normal} in the sense of Allcock
and is ``sparse'' in the sense that and two disjoint components are sufficiently
separated. With these assumptions $\pi_1(V\setminus S)$ 
is hyperbolic relative to the fundamental group
of the boundary components of a regular neighborhood of $S$ in $V$~\cite{BelHru}.
If $V$ is real hyperbolic, then $V\setminus S$ admits a complete finite volume
metric of $-1\le K<0$ by~\cite{AbrSch}.
\end{ex}

\begin{rmk}
A version of the claim in
Example~\ref{ex: rel hyp fin vol}(2)-(3) holds when
$V$ is an open complete finitely volume negatively pinched manifold.
\end{rmk}

\begin{rmk} 
Without the assumptions that the preimage of $S$ to $\widetilde V$
is normal and ``sparse'', it seems unlikely that  $\pi_1(V\setminus S)$ 
in Example~\ref{ex: rel hyp fin vol}(3)
is hyperbolic relative to some easy-to-understand peripheral subgroups, so we ask:
\end{rmk}

\begin{quest}
Let $V$ be a finite volume complete negatively pinched manifold,
and $S$ be an immersed, codimension two,
compact, totally geodesic submanifold.
Is $\pi_1(V\setminus S)$ acylindrically hyperbolic?
\end{quest}

If $G$ is a finitely generated group 
that is hyperbolic relatively to a finite family
of peripheral subgroups, then $G$ inherits the following properties of
its peripheral subgroups:
\begin{enumerate}
\item solvability of the word problem~\cite{Far-rel-hyp, Osi-rel-hyp} (which is a litmus test for decency
of a group).
\item solvability of conjugacy problem~\cite{Bum-conj}.
\item being fully residually hyperbolic~\cite{Osi-fill, GroMan-fill}; here
given a class of groups $\mathcal C$, a group $G$ is {\it fully residually} $\mathcal C$ 
if any finite subset of $G$ can be mapped injectively by a homomorphism of $G$ onto
a group in $\mathcal C$.
\item being biautomatic~\cite{Reb}.
\item finiteness of asymptotic dimension~\cite{Osi-asy-dim}
\item rapid decay property~\cite{DruSap-rapid-decay}
\item Tits alternative: a subgroup of a relatively hyperbolic group
that does not contain a non-abelian free subgroup is {\it elementary\,}, i.e.
virtually-$\mathbb Z$, finite, 
or contained in a peripheral subgroup~\cite{Tuk-tits}. 
\end{enumerate}

\begin{ex}
If $V$ and $S$ are as in Example~\ref{ex: rel hyp fin vol}(2),
then the peripheral subgroups in the relatively hyperbolic groups structure
on $\pi_1(V\setminus S)$ are the fundamental groups of circle bundles over
components of $S$. The peripheral subgroups
have solvable word and conjugacy problems, have finite asymptotic
dimension and rapid decay property, are biautomatic, residually hyperbolic, and
their non-virtually-abelian groups contain free nonabelian subgroups, 
see~\cite{Bel-ch-warp, Bel-rh-warp}.
Hence all these properties are inherited by $\pi_1(V\setminus S)$. 
\end{ex}

The fact that most real hyperbolic manifolds are a-K\"ahler~\cite{Gro-asy}
was used in~\cite{Bel-hyper-polyh} to show

\begin{thm}
$R_{W,B}$ is not homeomorphic to an open subset of a K\"ahler manifold
of real dimension $\ge 4$.
\end{thm}

In fact, ``not homeomorphic'' can be replaced with 
``not proper homotopy equivalent'' under a mild assumption 
on the strict hyperbolization block~\cite{Bel-hyper-polyh}.

\section{Zoo of finite volume rank one manifolds}
\label{sec: inf hom}

In this section we discuss examples and structure of
connected open complete finite volume manifolds of $K\le 0$ and rank one,
with a particular focus on manifolds of $K<0$. 

\begin{thm}\label{thm: rk1 acyl hyp}
If $V$ is a complete finite volume manifold of rank one, then
$\pi_1(V)$ is acylindically hyperbolic. 
\end{thm}
\begin{proof}
The universal cover $\widetilde V$ has rank one, and
Ballmann~\cite{Bal-book} proved that a lattice in the isometry
group of a rank one Hadamard manifold contains a rank one element
that lies in a noncyclic free subgroup, 
so Sisto's Theorem~\ref{thm: sisto rk1} applies.
\end{proof}

\begin{quest}
Does every complete rank one manifold with $K\le 0$ and 
$\mathrm{Inj}\,\mathrm{Rad}\to 0$ have
acylindically hyperbolic fundamental group?
\end{quest}

Kapovitch-Wilking's version of the Margulis lemma 
only calls for $K\ge -1$ on a ball of radius $1$, 
and that curvature bound can always be achieved by rescaling.
Since acylindically hyperbolic groups are not virtually nilpotent, 
Theorem~\ref{thm: wilk-kap margulis} immediately implies: 

\begin{cor}
If $\e$ is the constant in \textup{Theorem~\ref{thm: wilk-kap margulis}},
then a finite volume complete rank one manifold $M$ contains no point $p$ such that 
$K\ge -1$ on $B_p(1)$ and the inclusion $B_p(\e)\hookrightarrow M$ is $\pi_1$-surjective.
\end{cor}

Thus if $\pi_1(M)$ is ``concentrated'' on an $\e$-ball,
then $K$ blows up near that ball.

Known examples of finite volume complete manifolds of $K<0$ that admit
no finite volume metric of $-1\le K\le 0$ are based on 
Theorem~\ref{thm: Gromov simp vol}
that a compact boundary component of a manifold whose
interior has a complete metric of $\Ric\ge-(n-1)$
has zero simplicial volume. Indeed, Nguyen Phan~\cite{Pha-neg}
proved

\begin{thm}
\bf (Nguyen Phan)\ \it 
In each dimension $\ge 3$ there exists a
finite volume complete manifold of $K<0$ 
that is the interior of a compact manifold
whose boundary admits a real hyperbolic metric.
\end{thm} 

Recall that a closed real hyperbolic manifold has a
positive simplicial volume. 
The proof of the above theorem is by explicit construction; 
alternatively, it follows from Ontaneda's Theorem~\ref{thm-intro: ont}
by inserting in the cusp the warped product $\mathbb R\times_{e^r} B$ where
$B$ is any closed manifold of $K\le 0$ with nonzero simplicial volume.
In fact, this argument proves:

\begin{thm}\bf (Ontaneda)\ \it 
If a closed manifold of $K\le 0$ is diffeomorphic to the boundary
of a compact manifold, then it is diffeomorphic to the boundary 
of a compact manifold whose interior admits a complete finite volume
metric of $K\le -1$.
\end{thm} 

\begin{ex}
There are closed manifolds of $K\le 0$ that are not homeomorphic to
the boundaries of compact manifolds. One source of such examples is evenness
of the Euler characteristic of the boundary of a compact 
manifold~\cite[Corollary VIII.8.8]{Dol-book}. Examples of
closed manifold of $K\le 0$ with odd Euler characteristic include
suitable closed non-orientable hyperbolic surfaces, 
Mumford's complex hyperbolic surface of Euler characteristic 
$3$~\cite[Proposition2.2]{HerPau}, or their products.
\end{ex}

Iterated gluing along totally geodesic boundaries sometimes yields 
finite volume $M$ with infinitely generated fundamental group.
This phenomenon was discovered by Gromov~\cite[Chapter 11]{BGS}
who produced such $3$-dimensional graph manifolds with $K\in [-1, 0]$ and 
no local Euclidean de Rham factor.
Other examples with infinitely generated fundamental groups, due to 
Nguyen Phan~\cite{Pha-neg}, are 
finite volume manifolds with $K\le -1$ and infinitely many ends,
appearing in all dimensions $\ge 2$.

A related idea of Nguyen Phan~\cite{Pha-neg}
uses infinite cyclic covers
of closed manifolds of $K<0$ to produce two-ended
finite volume manifolds of $K<0$:

\begin{thm}
\label{thm: cyclic cover}
\bf (Nguyen Phan)\ \it 
If $\widehat L$ is an infinite cyclic cover of a closed manifold $L$ of $K<0$
of dimension $l\ge 3$, then $\widehat L$ admits a complete finite volume metric
metric of $K<0$.
\end{thm}

Here $\widehat L\to L$ 
is the covering that corresponds to the kernel of any epimorphism
$\pi_1(L)\to\mathbb Z$. A most famous example is when $\widehat L$ corresponds to
the fiber group in a closed hyperbolic $3$-manifolds that fibers over a circle.

\begin{ex} (of $\widehat L$ with infinitely generated fundamental group)
Suppose that $\pi_1(L)$ surjects onto a noncyclic free group $F_r$.
The kernel of any epimorphism $F_r\to\mathbb Z$ is infinitely generated,
so hence so is the kernel of the composite $\pi_1(L)\to F_r\to\mathbb Z$.
\end{ex}

\begin{quest}
Is $\pi_1(\widehat L)$ always infinitely generated when $\dim(\widehat L)>3$?  
\end{quest}

\begin{rmk}
If $\dim(\widehat L)\ge 6$ and $\widehat L$ is  homotopy equivalent to a finite
cell complex (or more generally is finitely dominated), 
then $L$ (smoothly) fibers over a circle, and the fiber is a closed 
aspherical manifold whose inclusion into $L$ corresponds homotopically
to the covering $\widehat L\to L$.
Indeed, Siebenmann's version~\cite{Sie-totwh} of Farrell's fibering obstruction 
lies in the group $\mathrm{Wh}(\pi_1(L))$, which is zero by~\cite{FJ-dynI}.
\end{rmk}

\begin{rmk} Another restriction on $\pi_1(\widehat L)$ is that its
outer automorphism group is infinite (as easily follows from
the fact that $\pi_1(\widehat L)$ has 
trivial centralizer in $\pi_1(L)$). Combining with the previous
remark we see that if $\dim(\widehat L)\ge 6$ and $\widehat L$ is finitely dominated,
then $\widehat L$ is homotopy equivalent to a closed aspherical manifold of dimension
$\ge 5$ whose fundamental group has infinite outer automorphism group
and embeds into the hyperbolic group $\pi_1(L)$;
it seems such a closed aspherical manifold cannot exist, so we ask:
\end{rmk}

\begin{quest}
Can $\widehat L$ ever be finitely dominated when $\dim(\widehat L)>3$?  
\end{quest}

\begin{ex}
\label{ex: phan surface cross R}
In dimension three Theorem~\ref{thm: cyclic cover} gives a finite volume metric of $K<0$
on the product of a closed hyperbolic surface with $\mathbb R$, and interestingly, 
the metric can be chosen so that the corresponding nonuniform lattice 
contains no parabolics, see~\cite{Pha-neg}.
\end{ex}

\begin{quest}
Is there a nonuniform lattice in the isometry group of a Hadamard manifold of dimension $>3$
that contains no parabolics?
\end{quest}

Borel-Serre~\cite{BorSer74} compute the $\mathbb Q$-rank of a finite volume complete locally symmetric
$n$-manifold $V$ as $n-\cd(\pi_1(V))$, where $\cd$ denotes the cohomological dimension;
alternatively, $\mathbb Q$-rank equals the dimension of an asymptotic cone of
$V$~\cite{Hat-asy-cone}, cf.~\cite{JiMcP}, and $\mathbb Q$-rank can also be defined
in terms of flats in $V$~\cite{Mor-book}. If $V$ has rank one (i.e. contains a rank one geodesic),
then the $\mathbb Q$-rank of $V$ equals $1$. One wonders whether any of these relations
between $\cd$, asymptotic cone, and absence of $2$-dimensional flats extend to 
finite volume complete manifolds of rank one.

\begin{quest}
Is the asymptotic cone of a complete 
finite volume rank one manifold of $K\le 0$ always a tree?
\end{quest}

\begin{quest}
\label{quest: values of cd}
What values does $\mathrm{cd}$ take on the fundamental groups of open
complete finite volume $n$-manifolds of $K\le -1$, rank one, or $K<0$? 
Is $\mathrm{cd}$ always equal $n-1$?
\end{quest}

To better understand the above question 
let us relate bounds on $\cd$ with the fundamental group at infinity
of an open aspherical $n$-manifold $M$. First note that
if $\cd(\pi_1(M))\le n-2$, then $M$ is one-ended~\cite[Proposition 1.2]{Sie-collar}
but the converse fails (think of the open M\"obius band).
In the simplest case when 
$M$ is the interior of a compact manifold, we get the following clean
statement: 

\begin{prop} Let $N$ be a compact aspherical $n$-manifold with boundary. 
Then $\mathrm{cd}(\pi_1(M))\le n-2$ if and only if $\d N$ is connected
and the inclusion $\d N\hookrightarrow N$ is $\pi_1$-surjective. 
\end{prop}
\begin{proof} Clearly $\cd(\pi_1(M))\le n-1$. Now Poincar\'e-Lefschetz 
duality in the universal cover~\cite[Corollary VIII.8.3]{Bro-book} implies
that $\cd(\pi_1(M))=n-1$ if and only if the boundary of the
universal cover of $N$ is not connected. The latter is equivalent
to ``either $\d N$ is not connected or $\d N\hookrightarrow N$ is not
$\pi_1$-surjective'' by elementary covering space considerations.
\end{proof}

Question~\ref{quest: values of cd} should be compared with
Nguyen Phan's Example~\ref{ex: phan surface cross R} of a finite volume
manifold of $K<0$ that is quite small homologically, and perhaps
there are even smaller examples. We finish with the following tantalizing

\begin{quest} Given $n>2$, does the interior of the 
$n$-dimensional handlebody admit a complete finite volume metric of $K\le 0$?
\end{quest}

\begin{ex}  
The interior of an odd-dimensional handlebody admits no
complete  finite volume metric of $-1\le K\le 0$. 
(The boundary of an odd-dimensional handlebody with $g$ handles
has Euler characteristic $2-2g$ while under our geometric assumptions
the Euler characteristic vanishes by Theorem~\ref{cor: cheeger-gromov}
and $g\neq 1$ because $\mathbb Z$ is not acylindrically hyperbolic).
\end{ex}

\small
\bibliographystyle{amsalpha}
\bibliography{wp-bangalore.bib}

\end{document}